\documentclass[a4paper,10pt]{article}
\usepackage{amsmath,amsthm,amssymb,amscd,epsfig,esint, mathrsfs,graphics,color}
\usepackage{verbatim}
\usepackage[english]{babel}
\newtheorem{thm}{Theorem}
\newtheorem{lem}[thm]{Lemma}
\newtheorem{prop}[thm]{Proposition}

\newtheorem{theor}[thm]{Theorem}

\newcommand{\norm}[1]{\left\lVert #1 \right\rVert}

\def\R{\mathbb R} 

\def\dt{\mathrm{d}}

\def\G{\mathcal G}

\def\F{\mathcal F}

\def\min{\operatorname{min}}
\def\max{\operatorname{max}}

\def\Id{\mathbf{Id}}

\def\loc{{\mathop\mathrm{\,loc\,}}}

\numberwithin{equation}{section}

\title{
\sc{ Very degenerate elliptic
equations under almost critical Sobolev regularity}}

\author{ Albert Clop, Raffaella Giova, Farhad Hatami, Antonia Passarelli di Napoli}

\begin{document}
\maketitle

\begin{abstract}
{We prove the local Lipschitz continuity and the higher differentiability of local minimizers of functionals of the form
$$\mathbb{F}(u,\Omega)=\int_{\Omega}\big( \F(x,Du)+f(x)\cdot
u(x)\Big)\,dx$$
with {non autonomous} integrand $\F(x,\xi)$ which is {degenerate} convex with respect to the gradient variable. The main novelty here is that the results are obtained assuming that the partial map $x\mapsto D_\xi\F(x,\xi)$ has weak derivative in the almost critical Zygmund class $L^n\log^\alpha L$ and the datum $f$ is assumed to belong to the same Zygmund class. }
\end{abstract}

\bigskip

{\footnotesize{ \emph{Mathematics Subject
Classification}. {35J47; 35J70; 49N60.}

}}

\bigskip 
\section{Introduction}
In this paper we deal with the
regularity properties of local minimizers of {non autonomous and} non homogeneous functionals of the form
\begin{equation}\label{modenergy}
\mathbb{F}(u,\Omega)=\int_{\Omega}\big( \F(x,Du)+f(x)\cdot
u(x)\Big)\,dx,
\end{equation}
where  $\Omega\subset\R^n$, $n\ge 2$ is a bounded open set,
$u:\Omega\to \R^N$, $N\geq 1$ , $\F: \Omega\times \R^{n\times
N}\to [0,+\infty)$ is a Carath\'eodory function with growth $p\geq 2$, {assumed to be  uniformly convex with respect to the gradient variable   \emph{only} at $\infty$}, and $f:\Omega\subset\R^n\to\R^N$ is a  given
datum. To be more precise, we shall assume that $\xi\mapsto \F(x,\xi) $ is convex and, {for an exponent $p\ge 2$}, satisfies the following {set of hypotheses.}
\begin{enumerate}
\item[{\bf{(F0)}}] There exist positive constants $\ell, L$ such that
$$\ell (|\xi|^p-1)\le \F(x,\xi)\le L(|\xi|^p+1)$$
for a.e. $x\in\Omega$ and every $\xi\in\R^{n\times N}$.
\end{enumerate}
{This is the unique assumption that has to hold true on the whole $\R^{n\times N}$. From now on we list the assumptions that need to hold only at $\infty$, i.e. outside a ball of radius $\mathcal{R}$. }
\begin{enumerate}
\item[{\bf{(F1)}}] There exists
$F:\overline\Omega\times[\mathcal{R},\infty)\to [0,+\infty)$ such that
$$\F(x,\xi)=F(x,|\xi|)$$
for a.e. $x\in\Omega$ and for every $\xi\in\R^{n\times N}\setminus
B_{\mathcal{R}}(0)$.
\item[{\bf{(F2)}}] {The partial map $\xi\mapsto \F(x,\xi) $ is $C^2(\R^{n\times N}\setminus
B_{\mathcal{R}}(0))$ and }there exists $\nu>0$ such that
$$\langle D_{\xi\xi}\F(x,\xi)\lambda,\lambda\rangle\geqslant\nu|\xi|^{p-2}|\lambda|^{2},$$
 for a.e. $x\in\Omega$, every
$\lambda\in \R^{n\times N}$ and every $\xi\in\R^{n\times N}\setminus
B_{\mathcal{R}}(0)$.
\item[{\bf{(F3)}}] There exists a positive constant $L_{1}$ such that
$$|D_{\xi\xi}\F(x,\xi)|\leqslant L_{1}|\xi|^{p-2}$$
for a.e. $x\in\Omega$  and  every $\xi\in\R^{n\times N}\setminus
B_{\mathcal{R}}(0)$.
\item[{\bf{(F4)}}] There exist and exponent $\alpha>0$ and a non-negative function $k\in L^n\log^\alpha L_{\loc}(\Omega)$  such that
$$|D_{x\xi}\F(x,\xi)|\leqslant k(x)|\xi|^{p-1},$$
for a.e. $ x\in\Omega$ and for every $\xi\in\R^{n\times N}\setminus
B_{\mathcal{R}}(0)$.
\end{enumerate}

\noindent There is no loss of generality in assuming that $\mathcal{R}=1$, what we will do in what follows.

\noindent The functionals with degenerate convexity have
attracted a wide interest in the last few years. Among the reasons, they arise in
problems of optimal transport with congestion effects and in the
construction of Wardrop equilibriums in traffic problems (see for instance Carlier, Jim\'enez and Santambrogio \cite{CJS}). In
particular, the regularity of minimizers, and more specifically their
Lipschitz continuity, allows the application of the Di Perna - Lions
theory to describe the equilibrium configurations. This
connection  was first described by Brasco, Carlier and Santambrogio
\cite{BCS}, and has partially motivated a research line which is nowadays very active. 
\\
In this setting, the local {Lipschitz regularity} of the minimizers  was proven in \cite[Theorem
5.2]{BCS} for $f\in C^\alpha$ and $\alpha>0$ , and in {\cite[Theorem
2.1]{B}} for $f\in L^s$, $s>n$ { in case the energy density has the following special structure
$$ \F(x,\xi)=(|\xi|-1)_+^p, \qquad\qquad p\ge 2.$$} 
\\
It is worth mentioning that the regularity of minimizers of non autonomous widely degenerate functionals is challenging also from the theoretical point of view and many contributions to its study  are available. 
\\
Actually, they fit into the broader contest of asymptotically convex functionals, whose study started with the pioneering paper by Chipot and Evans (\cite{CE}) concerning the {homogeneous}, autonomous, quadratic growth case. Later on, still for the {homogeneous}, autonomous case, the Lipschitz continuity of local minimizers of asymptotically convex functionals has been established for the superquadratic growth  (\cite{giamod86}), in the subquadratic growth case (\cite{LPV}).
Since then, many contributions to the regularity theory of asymptotically convex functionals have been established. 
Among the others we quote the results in \cite{CupGuiMas, fossnapoliverde1, fossnapoliverde2, fossgoodrich1, fossgoodrich2, goodrich4, goodrich1, goodrich2, Kronz,  R, scheven1, scheven2}.
\\
We'd like to recall that, in the non autonomous
homogeneous case, i.e. $f\equiv 0$, Fonseca, Fusco and Marcellini
proved in \cite{FFM} that local minimizers of $\mathbb{F}(u,\Omega)$  are Lipschitz continuous
if
$$|D_{x\xi}\F(x,\xi)|\leq k(x)\,(1+|\xi|)^{p-1}\hspace{1cm}(x,\xi)\in\Omega\times \R^{n\times N}$$
for some function $k\in L^\infty$, i.e. under a Lipschitz regularity of the partial map $x\mapsto D_{\xi}\F(x,\xi)$. More recently, still for $f\equiv
0$, in \cite{EMM} this result was extended to the case $k\in
L^s_{\loc}(\Omega)$ for some $s>n$ {(see also \cite{EMM16, EMM18})}.
\\
Especially for applications, it is desirable to understand what is
the situation under  assumptions on  the map $x\mapsto D_{\xi}\F(x,\xi)$ weaker than the H\"older  continuity.
 We know,
though, that if $s=n=2$ then $\nabla u\in L^\infty$ may fail even
in the linear uniformly elliptic setting \cite{CFMOZ,APdN2}, while
in the case $n>2$ some regularity results under a Sobolev or
Orlicz Sobolev assumption on the coefficients, that doesn't imply
their H\"older continuity, can be found in \cite{goodrich3, Gio, Giova2,
GP2016, APdN1, APdN-Levico}.
\\
The case  $k\in
L^n_{\loc}(\Omega)$ has been studied in \cite{CGGP} for homogeneous functionals with degenerate convexity, but obtaining
only $W^{1,t}$ regularity for every $t\in[p,\infty)$.
\\
{As far as we know, no Lipschitz regularity results are available for local minimizers of this class of functionals  under weaker assumptions on the coefficients and the datum and this motivated our study. }

\noindent Here we will show that {the gradient of the local minimizers of $\mathbb{F}(u,\Omega)$ is locally bounded in $\Omega$} under a regularity assumption on the map $x
\rightarrow D_{\xi}f(x,\xi)$ and the  datum $f$
 that, roughly speaking, will be intermediate between  $L^n$ and $L^{ s}$, with ${s}>n$. \\
 Actually,  we are
able to prove that {in order to obtain the local Lipschitz regularity of the  minimizers it suffices to} assume that both the weak derivative with respect to
$x$ of $D_{\xi}f(x,\xi)$ and the datum $f$ belong to the Orlicz
Zygmund class $L^n \log^{\alpha} L$, with $\alpha>0$ sufficiently
large with respect to the dimension $n$. More precisely, our
main result is the following

\begin{thm}\label{thmain1}
Let $u\in W^{1,p}_{\loc}(\Omega)$ be a local minimizer of the
functional $\mathbb{F}(u,\Omega)$ in \eqref{modenergy}, and assume
that the energy density $\F(x,\xi)$ satisfies assumptions
{\bf{(F0)--(F4)}}. If $f\in L^n\log^\alpha L_{\loc}(\Omega)$, with $\alpha>3n$ and where
$\alpha$ is the exponent appearing in {\bf{(F4)}}, then $Du\in L^\infty_{\mathrm{loc}}(\Omega)$ and the following estimate
\begin{align*}
\sup\limits_{B_{\rho}}|Du| \leqslant
\hat{C}{\left(\int_{B_r}(1+|Du|)^p\,dx,\right)^\frac1p}
\end{align*}
holds for every balls $B_{\rho}\subset B_r\Subset\Omega$ and for a
constant  $\hat{C}=\hat{C}(n,p,\alpha,\nu,L_1,\rho, r,
||k||_{L^n\log^\alpha L(B_r)},||f||_{L^n\log^\alpha L(B_r)})$.
\end{thm}
Note that, {by virtue of the embedding Theorems in the Orlicz-Sobolev setting (\cite{Cianchi}),} our assumption on the partial map $x\mapsto D_\xi\F(x,\xi)$ implies the continuity of the coefficients but not their H\"older continuity.

\noindent
As usual, the proof of this Theorem is divided in two parts: an a priori $L^\infty$ estimate, and an approximation argument.
\\
{The a priori estimate is achieved by the use of test functions that vanish on the set where the functional in not uniformly convex.
The  novelty here is the use of a suitable Young's inequality in Orlicz space and an appropriate modification of the Moser iteration argument.
\\
The size of the exponent $\alpha$ in assumption {\bf{(F4)}} is used in order to let this iteration argument work.
\\
Once the a priori estimate is established, we perform an approximation procedure with a sequence of functionals regular enough to apply the a priori estimate to their minimizers.
\\
The lack of the uniqueness of the minimizers, due to the lack of uniform convexity of the integrand, could effect the approximation argument. We overcome this difficulty, as done in \cite{CCG},  by adding to the approximating functionals a penalization term that forces the approximating sequence to converge to an arbitrarily fixed local minimizers.}
\\
{It is worth mentioning that in order to start the Moser iteration procedure, we need to estimate quantities involving the second derivatives of the local minimizers. Therefore,
 another result that we are able to establish here consists in showing } that minimizers turn out to gain one degree of differentiability, in the weak sense, away from the degeneracy set. This can be quantitatively stated as follows. Denoting by
\begin{equation}\label{funzioneg}
\mathcal{G}(t):=1+\int_0^t(1+s)^{\frac{p-4}{2}}s\,ds,	
\end{equation}
{we have the following higher differentiability result}.

{\begin{thm}\label{thmain2}
Let $u\in W^{1,p}(\Omega; \mathbb{R}^{n\times N})$ be a  local
minimizer of the functional $\mathbb{F}(\cdot,\Omega)$ in
\eqref{modenergy}, and assume that the energy density  $\F(x,\xi)$ satisfies assumptions  {\bf{(F0)-(F4)}}. If $f\in  L^{n}\log^\alpha L_{\loc}(\Omega)$ with $\alpha>3n$, then
$$
\mathcal{G}((|Du|-1)_+)\in W^{1,2}_{\mathrm{loc}}(\Omega)
$$
and the following Caccioppoli type inequality holds,
\begin{equation}\label{apri}
\int_{B_\rho(x_0)}|D(\mathcal{G}((|Du|-1)_+))|^2 \le
\widetilde{C}
     \int_{B_{r}(x_0)}\left(1+|Du|^p\right)\,dx,
  \end{equation}
for every balls $B_\rho\subset B_{r}(x_{0})\Subset\Omega$, for
some $\widetilde{C}=\widetilde{C}(n,p,\alpha,\nu,L_1,\rho, r,
||k||_{L^n\log^\alpha L(B_r)},||f||_{L^n\log^\alpha L(B_r)})$.
\end{thm}}
The proof of Theorem \ref{thmain2} is established  as before  combining an  a priori estimate with an approximation argument. 
\\
This time the regularity of the approximating minimizers  is transferred to the limit by the use of a measure theory result ( see Proposition \ref{anto} below) that is an improved version of Proposition 6 in \cite{CGGP2}.

\noindent
{{Let us mention that in   equation \eqref{apri}, the term on the left hand side is equivalent to
$$
\int_{B_\rho(x_0)}|D(\mathcal{G}((|Du|-1)_+))|^2=\int_{B_\rho(x_0)}
(|Du|-1)_+^2\,|Du|^{p-4}\,|D^2u|^2\,dx
$$
so that the above result is, in fact, a weighted  bound for $D^2u$ with the weight $(|Du|-1)_+^2\,|Du|^{p-4}$.}}
\\
\\
The paper is structured as follows. {In Section \ref{preliminaryresults} we recall some preliminaries. In Section  \ref{slipschitz} we prove {both} the interior a priori Lipschitz  {and the a priori second order Sobolev } estimates. In Section \ref{appsubsection} we prove {the Lipschitz regularity of the minimizers of $\mathbb{F}(u,\Omega)$, i.e.} Theorem \ref{thmain1}. In Section \ref{ssobolev}  we prove the interior second order Sobolev {regularity of the minimizers of $\mathbb{F}(u,\Omega)$}, i.e. Theorem \ref{thmain2}. }

\section{Preliminary results}\label{preliminaryresults}

We will  follow the usual convention and denote by $c$ or $C$ a
general constant that may vary on different occasions, even within
the same line of estimates. Relevant dependencies on parameters and
special constants will be suitably emphasized using parentheses or
subscripts. All the norms we use  will be the standard Euclidean
ones and denoted by $|\cdot |$ in all cases. In particular, for
matrices $\xi$, $\eta \in \R^{n\times m}$ we write $\langle \xi,
\eta \rangle : = \text{trace} (\xi^T \eta)$ for the usual inner
product of $\xi$ and $\eta$, and $| \xi | : = \langle \xi,
\xi\rangle^{\frac{1}{2}}$ for the corresponding euclidean norm. By
$B_r(x)$  we will denote the ball in $\mathbb{R}^n$ centered at $x$
of radius $r$. The integral mean of a function  $u$ over a ball
$B_r(x)$  will be denoted by $u_{x,r}$, that is
$$ u_{x,r}:=\frac{1}{|B_r(x)|}\int_{B_r(x)}u(y)\,dy,$$
where $|B_r(x)|$ is the Lebesgue measure of the ball in
$\mathbb{R}^{n}$. If no confusion  arises, we shall omit the
dependence on the center.

The following lemma is an important application in the so called
hole-filling method. Its proof can be found for example in
\cite[Lemma 6.1]{gi} .
\medskip
\begin{lem}\label{iter} Let $h:[r, R_{0}]\to \mathbb{R}$ be a nonnegative bounded function and $0<\vartheta<1$,
$A, B\ge 0$ and $\beta>0$. Assume that
$$
h(s)\leq \vartheta h(t)+\frac{A}{(t-s)^{\beta}}+B,
$$
for all $r\leq s<t\leq R_{0}$. Then
$$
h(r)\leq \frac{c A}{(R_{0}-r)^{\beta}}+cB ,
$$
where $c=c(\vartheta, \beta)>0$.
\end{lem}
In what follows, for each $\gamma\ge
0$, we will denote
\begin{equation}\label{phigamma}
\Phi(t)=\Phi_\gamma(t)=\frac{t^{2}}{(1+t)^{2}}(1+t)^{\gamma}.
\end{equation}
For such $\Phi$, one can easily check that
\begin{equation}\label{eq48}
\aligned
t\Phi^{\prime}(t)
&\leq 2(1+\gamma)\Phi(t),
\endaligned
\end{equation}
{and also that, since $$\Phi'(t)\le 2(1+\gamma)t(1+t^2)^{\gamma-2},$$
we have
\begin{equation}\label{phiprimo}
	\Phi'(t)\le 2^{\gamma+1}(1+\gamma)\qquad\qquad \text{for every}\,\, t\in (0,1).
\end{equation}}
We  introduce the following notation for the positive part of
$|Du|-1$,
$$P=(|Du|-1)_+=\max\{|Du|-1,0\}$$
so that
\begin{equation}\label{equality}
\aligned DP&=\chi_{\{|Du|>1\}}\cdot\dfrac{Du}{|Du|}\cdot
D^{2}u.\endaligned
\end{equation}
\subsection{Orlicz Spaces}\label{orl sec}
In this section we recall some basic properties of Orlicz spaces
(for more details we refer to  \cite{Adams}).

\medskip

Let $\Psi: [0,\infty) \rightarrow [0,\infty)$ be a Young function,
that is $\Psi(0)=0$, $\Psi$ is increasing and convex. If $\Omega$ is
a open subset of $\R^n$, we define  the Orlicz space
$L^\Psi(\Omega)$ generated by the Young function $\Psi$ as the set
of the measurable functions $u:\Omega \rightarrow \mathbb R$ such
that
\[
 \int_\Omega \Psi
\left( \frac {|u|} {\lambda} \right)\,dx< \infty,
\]
for some $\lambda>0$. Once equipped with the Luxemburg norm
\[\|u\|_{L^\Psi (\Omega)}=\inf\left\{\lambda>0\ :\ \int_\Omega \Psi
\left( \frac {|u|} {\lambda} \right) \,dx\leqslant 1\right\}\]
$L^\Psi(\Omega)$ is a Banach space. We
define the space $WL^\Psi(\Omega)$ as the set
\begin{equation*}
WL^\Psi(\Omega)= \left\{ u \in W^{1,1} (\Omega): \, |\nabla u| \in
L^\Psi   (\Omega) \right\}.
\end{equation*}
The Zygmund space $L^p\log^\alpha L (\Omega)$, for $1\leqslant
p<\infty$, $\alpha\in\R$ ($\alpha\geqslant 0$ for $p=1$), is defined
as the Orlicz space $L^\Psi(\Omega)$ generated by the Young function

\begin{equation}\label{log}
\Psi(t) \simeq t^p \log^\alpha(e+t)\, \qquad
\quad \text{for every $t\ge t_0 \ge 0$} \,.
\end{equation}
Therefore,   a measurable function $u$ on $\Omega$ belongs to  $  L
^p\log^\alpha  L (\Omega)$ if
\[\int_\Omega |u|^p\log^\alpha(e+|u|)\,dx<\infty\,,\]
and we {recall that the quantity} 
\[
[u]_{p,\alpha}=\biggm[ \int
|u|^p\log^\alpha\left(e+\frac{|u|}{\|u\|_p}\right)\,dx\biggm]^{\frac{1}{p}}
\]
 is equivalent to the Luxemburg norm for $\alpha\ge 0$.
{For further needs, we record that the function $\Psi(t)=t^p \log^\alpha(e+t)$, $p>1$, $\alpha\in \mathbb{R}$, satisfy the so called $\Delta_2$ and $\nabla_2$ conditions. This is equivalent   to the fact that 
$$\frac{\Psi(t)}{t^{p-\varepsilon}}\nearrow\qquad\qquad \frac{\Psi(t)}{t^{p+\varepsilon}}\searrow\qquad\qquad \text{for every}\,\, \varepsilon>0,$$
and so, it  can be easily checked that
$$\Psi(\lambda t)\ge \lambda^{p+\varepsilon}\Psi(t)\qquad\qquad \text{for every }\lambda\in (0,1)$$
and
$$\Psi(\mu t)\ge \mu^{p-\varepsilon}\Psi(t)\qquad\qquad \text{for every }\mu>1$$
From this and the definition of Luxenburg norm, we deduce that
\begin{equation*}
	1=\int_\Omega \Psi\left(\frac{|f|}{||f||_{\Psi}}\right)\,dx\ge \left(\frac{1}{||f||_{\Psi}}\right)^{p+\varepsilon}\int_\Omega \Psi(|f|)\qquad \text{if}\,\, ||f||_{\Psi}\ge 1
\end{equation*}
and 
\begin{equation*}
	1=\int_\Omega \Psi\left(\frac{|f|}{||f||_{\Psi}}\right)\,dx\ge \left(\frac{1}{||f||_{\Psi}}\right)^{p-\varepsilon}\int_\Omega \Psi(|f|)\qquad \text{if}\,\, ||f||_{\Psi}\le 1
\end{equation*}
We deduce that
\begin{equation}\label{normint}
 \int_\Omega |f|^p\log^\alpha(e+|f|)\,dx	\le ||f||_{\Psi}^{\vartheta}
\end{equation}}
for some $\vartheta=\vartheta(p)>0$.


\section{The a priori $L^\infty$ estimate }\label{slipschitz}


\noindent {This section is devoted to the proof of the a priori estimates that  will be the crucial steps
in the proofs of both Theorems \ref{thmain1} and \ref{thmain2} }. The precise
statement is the following one.

\begin{theor}\label{t:Apriori2}
 Assume {\bf{(F0)}}--{\bf{(F4)}} hold, and let $f\in L^n\log^\alpha L_{\loc}(\Omega)$ where $\alpha>3n$ is the exponent appearing in {\bf{(F4)}}.
Fix  a ball $B_{r}(x_0)\Subset \Omega$, and functions $u,
\bar{u}\in W^{1,p}(B_{r}(x_0);\mathbb{R}^N)$, and define
\[{\widetilde{\mathbb{F}}}(v;B_r(x_0)):=\mathbb{F}(v;B_{r}(x_0))+\int_{B_{r}(x_0)}\arctan(|v-\bar{u}(x)|^2)\,dx,\]
where $\mathbb{F}$ is defined in \eqref{modenergy}.
Let $v\in u+W_0^{1,p}(B_{r}(x_0);\mathbb{R}^N)$ be a  minimizer of
${\widetilde{\mathbb{F}}}$, {satisfying
\[v\in W^{2,2}_{\mathrm{loc}}(B_{r}(x_0);\mathbb{R}^N)\cap W^{1,\infty}_{\mathrm{loc}}(B_{r}(x_0);\mathbb{R}^N)\quad \text{ and}\quad
 |Dv|^{{p-2}}|D^2v|^2\in
L^1_{\mathrm{loc}}(B_{r}(x_0)).\]} Then, for every
$B_{\bar{r}}(\bar{x})\Subset B_r(x_0)$,  every $0<\rho<r'\le
\bar{r}$ 
\begin{equation}\label{apr}
\sup\limits_{B_{\rho}(\bar{x})}|{Du}| \leqslant C
\bigg[\int_{B_{r'}(\bar{x})}(1+|Du|)^p\dt x\bigg]^{\frac{1}{p}},	
\end{equation} for
some  constant $C=C(n,N,p,L_1,\nu,\rho, r',||k||_{L^n\log^\alpha
L(B_{r'})},||f||_{L^n\log^\alpha L(B_{r'})})$. Moreover, one has
\begin{equation}\label{apriori3}
\int_{B_{\rho}(\bar{x})}
\frac{(|Du|-1)_+^2}{(1+(|Du|-1)_+)^2}\,|Du|^{p-2}|D^2u|^2\,dx\leq C
\int_{B_{r'}(\bar{x})}\left(1+|Du|^p\right)\,dx,\end{equation} for some
$C=C(n,N,p,L_1,\nu,\rho,r',||k||_{L^n\log^\alpha
L(B_{r'})},||f||_{L^n\log^\alpha L(B_{r'})})$.
 \end{theor}

\noindent  For the
proof of the above result, the integral  $
\int_{B_{r}(x_0)}\arctan(|v-\bar{u}|^2)\,dx$  is a perturbation of
$\mathbb{F}(v;B_{r}(x_0))$ that provides no difficulties. Indeed,
denoting $g(x,v):=\arctan(|v-\bar{u}(x)|^2)$, we have that $g$ and
its derivatives $g_{v^{\alpha}}$, $\alpha=1,\ldots,m$,  are bounded.
Thus, for the sake of clarity, we prefer to drop this perturbation
term, and to state, and prove an a priori estimate for local
minimizers of $\mathbb{F}(\cdot;\Omega)$ only, as done in \cite{CGGP2}, see Theorem
\ref{Apriori} below.

\begin{thm}\label{Apriori}
Let  $ \mathcal{F}(x,\xi)$ satisfy assumptions
{\bf{(F0)}}--{\bf{(F4)}}, and let $f\in L^n\log^\alpha
L_{\loc}(\Omega)$ where $\alpha>3n$ is the exponent appearing in {\bf{(F4)}}.  Assume that $u\in
W^{2,2}_{\loc}(\Omega,\R^{N})\cap
W^{1,\infty}_{\loc}(\Omega,\R^{N})$ is a local minimizer of the
functional $\mathbb{F}(\cdot;\Omega)$, and that
$|Du|^{p-2}|D^2u|^2\in L^1_{\mathrm{loc}}(\Omega)$. Then the
estimate
\begin{equation}\label{apriorisup}
\sup_{B_{\rho}}|{Du}| \leqslant \hat{C}\bigg[\int_{B_R}
(1+|Du|)^p\dt x\bigg]^{\frac{1}{p}},
\end{equation}
holds  for every concentric balls $B_\rho \subset B_R\Subset
\Omega$. Moreover, the following second order Caccioppoli type inequality
\begin{equation}\label{apriori3b}
\int_{B_{\rho}}
\frac{(|Du|-1)_+^2}{(1+(|Du|-1)_+)^{2}}\,|Du|^{p-2}|D^2u|^2\,dx
   \leq \hat{C}
     \int_{B_{2R}}\left(1+|Du|^p\right)\,\,dx,
  \end{equation}
holds for some
$\hat{C}=\hat{C}(n,N,p,L_1,\nu,\rho,R,||k||_{L^n\log^\alpha
L(B_{2R})},||f||_{L^n\log^\alpha L(B_{2R})})$ and for every
concentric balls $B_\rho \subset B_R\subset B_{2R}\Subset \Omega$.
\end{thm}

\begin{proof}
We will prove the theorem in 3 steps.

\noindent \textbf{Step 1. } The first step is to prove that for every
$\gamma\geqslant 0$ and for every non-negative  function $\eta\in C^{\infty}_{0}(\Omega)$, there exists a positive constant $C=C(n,N,\nu,p,L_1)$ such that
\begin{equation}\label{eq70}
\aligned
\int_{\Omega}\eta^2 \Phi(P)&\,|Du|^{p-2}|D^2u|^2\,dx\leq C(\gamma+1)^2\int_{\{x\in\Omega:\,\,|Du|>1\}}\eta^2k^2|Du|^{\gamma+p}\,dx \\
&+   C\int_{\{x\in\Omega:\,\,|Du|>1\}}|D\eta|^2|Du|^{\gamma+p}\,dx
+C(\gamma+1)^2\,\int_{\{x\in\Omega:\,\,|Du|>1\}}\eta^{2}|f|^{2}|Du|^{\gamma}\,dx,
\endaligned
\end{equation}
where $\Phi$ is the function defined at \eqref{phigamma} and
$P=(|Du|-1)_+$. 
\\
Since $u$ is a local
minimizer of $\mathbb{F}(\cdot;\Omega)$, it satisfies the following
integral identity
$$\int_{\Omega}\langle D_\xi\F(x,Du),D\psi\rangle=\int_{\Omega}f\cdot\psi\quad\quad\forall\psi\in C^{\infty}_{0}(\Omega,\R^{N}).$$
By our assumptions on $u$ and a standard approximation argument, we
can choose
$$\psi\equiv \sum_sD_{x_{s}}\Big(\eta^{2}\cdot\Phi(P)\cdot D_{x_{s}}u\Big),$$
where  $\eta\in C_0^\infty(\Omega)$. Such a choice, together with an
integration by parts in the left hand side of previous identity,
yields
\begin{equation}\label{eq71}
-\sum_s\int_{\Omega}\!\!\left\langle D_{x_s
\xi}\F(x,Du)+D_{\xi\xi}\F(x,Du)\cdot D_{x_{s}}Du
,D\big(\eta^{2}\Phi(P) D_{x_{s}}u\big)\right\rangle
=\sum_s\int_{\Omega}f\eta^{2}\Phi(P) D_{x_{s}}u
\end{equation}
We now use the product rule to calculate the derivatives of
$\eta^2\cdot\Phi(P)\cdot D_{x_{s}}u$. This converts \eqref{eq71}
into
$$I_{1}+I_{2}+I_{3}+I_{4}+I_{5}+I_{6}+I_{7}+I_8+I_9=0,$$
where
\begin{eqnarray*}
I_{1}&=&2\sum_s\int_\Omega\langle  D_{\xi\xi}\F(x,Du)\cdot
D_{x_{s}}Du, D\eta\cdot D_{x_{s}}u\rangle\, \eta\,
\Phi(P)\,dx,\cr\cr I_{2}&=&\sum_s\int_\Omega\langle
D_{\xi\xi}\F(x,Du)\cdot D_{x_{s}}Du, D_{x_{s}}Du\rangle\,
\eta^{2}\,\Phi(P)\,dx,\cr\cr I_{3}&=&\sum_s\int_\Omega\langle
D_{\xi\xi}\F(x,Du)\cdot D_{x_{s}}Du,\Phi'(P)\,D_{x_{s}}P\cdot
D_{x_{s}}u\rangle\,\eta^{2}\,dx,\cr\cr
I_{4}&=&2\sum_s\int_\Omega\langle D_{x_{s}\xi}\F(x,Du),D\eta\cdot
D_{x_{s}}u\rangle\, \eta\, \Phi(P)\,dx,\cr\cr
I_{5}&=&\sum_s\int_\Omega\langle
D_{x_{s}\xi}\F(x,Du),D_{x_{s}}Du\rangle\,
\eta^{2}\,\Phi(P)\,dx,\cr\cr I_{6}&=&\sum_s\int_\Omega\langle
D_{x_{s}\xi}\F(x,Du),\Phi'(P)\,D_{x_{s}}P\cdot D_{x_{s}}u \rangle\,
\eta^{2}\,dx,\cr\cr I_{7}&=&2\sum_s \int_\Omega f\,\eta\,
\Phi(P)\,D_{x_s}\eta\cdot D_{x_{s}}u \,dx,\cr\cr
I_8&=&\sum_s\int_\Omega f\,\eta^{2}\,\Phi(P)\, D_{x_s x_s}u
\,dx,
\cr\cr I_9&=&\sum_s\int_\Omega f\,\eta^2\,\Phi'(P)\,D_{x_{s}}u
\cdot D_{x_s}P  \,dx.
\end{eqnarray*}
We will estimate each term separately. It is worth pointing out that the integrals $I_i$, with $i=1,\dots,6$
will be estimate with arguments similar to those in  \cite{EMM}, \cite{CGGP}. We will report here for the sake of completeness.\\
\\
For the estimate of $I_{1}$, we use assumption {\bf{(F3)}} and
Young's inequality as follows
\begin{align}\label{eq46.1}
|I_{1}|\leqslant&
2L_1\sum_s\int_\Omega \eta\,|D\eta|\,|Du|^{p-2}\,|D_{x_{s}}Du|\,\,|D_{x_{s}}u|\,\Phi(P)\,\,dx\nonumber\\
\leqslant& \varepsilon\int_\Omega \eta^{2}
|Du|^{p-2}|D^2u|^{2}\Phi(P)\,dx+C_\varepsilon(L_{1})\int_\Omega
|D\eta|^{2}|Du|^{p}\Phi(P)\,dx,
\end{align}
where $\varepsilon>0$ will be chosen later.
In order to estimate $I_{4}$, we use assumption {\bf{(F4)}} and
Young's inequality as follows,
\begin{equation}\label{eq47.3}
\aligned |I_{4}|
&\leqslant 2\int_\Omega \eta\,|D\eta|\,k\,|Du|^{p}\,\Phi(P)\,dx \\
&\leqslant C\int_\Omega \eta^2\,
k^2\,|Du|^{p}\,\Phi(P)\,dx+C\int_\Omega  |D\eta|^2
\,|Du|^{p}\,\Phi(P)\,dx.
\endaligned
\end{equation}
We estimate $I_{5}$, using {\bf{(F4)}} and  Young's inequality
again. Indeed
\begin{equation}\label{eq47.4}
\aligned
|I_{5}|&\leqslant\int_\Omega \eta^{2} k\,|Du|^{p-1}\,|D^2u|\,\Phi(P)\,dx\\
&\leqslant \varepsilon\int_\Omega
\eta^{2}\,|Du|^{p-2}\,|D^{2}u|^{2}\,\Phi(P)\, \,dx
+C_\varepsilon\int_\Omega \eta^{2} \,k^{2}\,|Du|^{p} \,\Phi(P)\,dx.
\endaligned
\end{equation}
For the estimate of $I_{6}$, again by virtue of assumption
{\bf{(F4)}}, we have

\begin{equation}\label{eq47.1}
\aligned |I_{6}|
&\leqslant \sum_s\int_\Omega \eta^2\,k\,|Du|^{p-1}\,|D_{x_{s}}u|\,\Phi'(P)\,|D_{x_{s}}P|\,dx \\
&\leqslant C\int_\Omega \eta^2\,k\,|Du|^{p}\,\Phi'(P)\,|DP|\,dx \leqslant C\int_\Omega \eta^2\,k\,|Du|^{p}\,\Phi'(P)|D^2u|\\
&= C\int_{\{|Du|\ge 2\}
}\eta^2\,k\,|Du|^{p}\,\Phi'(P)|D^2u|+C\int_{\{1<|Du|< 2\}
}\eta^2\,k\,|Du|^{p}\,\Phi'(P)|D^2u|,
\endaligned
\end{equation}
where we used the equality in \eqref{equality}. Noting that
\begin{equation}\label{est2}
    |Du|=(|Du|-1)_++1\le 2(|Du|-1)_+\quad\text{on the set}\,\,\, \{|Du|\ge 2\}
\end{equation}
and, recalling \eqref{eq48}, we can estimate the first integral in
the right hand side of previous inequality as follows
\begin{equation}\label{I6b}
\aligned
&C\int_{\{|Du|\ge 2\} }\eta^2\,k\,|Du|^{p}\,\Phi'(P)|D^2u|\le 2C \int_{\{|Du|\ge 2\} }\eta^2\,k\,|Du|^{p-1}\,P\Phi'(P)|D^2u|\\
&\le 4C(1+\gamma)\int_{\{|Du|\ge 2\} }\eta^2\,k\,|Du|^{p-1}\,\Phi(P)|D^2u|\\
&\le \varepsilon\int_{\{|Du|\ge 2\}
}\eta^2\,|Du|^{p-2}\,\Phi(P)|D^2u|^2+C_\varepsilon(1+\gamma)^2\int_{\{|Du|\ge
2\} }\eta^2\,k^2\,|Du|^{p}\,\Phi(P).
\endaligned
\end{equation}
 After setting $ \Gamma_{\gamma} =2(1+\gamma)>0$,  we multiply and divide the last integrand in \eqref{eq47.1}
 by $ \left(\frac{\delta +P}{\Gamma_\gamma}\right)^{1/2}$ with $0<\delta<1$, and use Young's inequality, thus
 obtaining
\begin{eqnarray}\label{eq52}
&&C\int_{\{1<|Du|< 2\} }
\eta^{2}\,\Phi'(P)\,\bigg\{\frac{\delta+P}{\Gamma_{\gamma}}
\,|Du|^{p-2}\,|D^{2}u|^{2}\bigg\}^{\frac{1}{2}}\times \bigg\{
\frac{\Gamma_{\gamma}}{\delta+P}\,k^{2}\,|Du|^{p+2}\bigg\}^{\frac{1}{2}}\,dx\cr\cr
&\leqslant&
 \varepsilon \int_{\{1<|Du|< 2\} } \eta^{2}\,\Phi'(P)\,\frac{\delta+P}{\Gamma_\gamma}\,|Du|^{p-2}\,|D^{2}u|^{2}\,dx\cr\cr
 &&+C_\varepsilon\int_{\{1<|Du|< 2\} }\eta^{2}\,k^{2}\,\Phi'(P)\,\frac{\Gamma_\gamma}{\delta+P}\, |Du|^{p+2}\,dx\cr\cr
 &\le&
\dfrac{\varepsilon}{\Gamma_{\gamma}}\int_{\{1< |Du|< 2\}}\eta^2\,\Phi'(P)\,P\,|Du|^{p-2}\,|D^{2}u|^2\,dx\nonumber+\frac{\varepsilon\delta}{\Gamma_{\gamma}}\int_{\{1< |Du|< 2\}}\eta^2\,\Phi'(P)\,|Du|^{p-2}\,|D^{2}u|^2\,dx\nonumber\\
 &&+C_\varepsilon\int_{\{1<|Du|< 2\} }\eta^{2}\, k^{2}\,\Phi'(P)\,\frac{\Gamma_\gamma}{\delta+P}\,|Du|^{p+2}\,dx\cr\cr
&\leqslant& \varepsilon\int_{\{1< |Du|<
2\}}\eta^2\,\Phi(P)\,|Du|^{p-2}\,|D^{2}u|^2\,dx+\varepsilon\delta\tilde C_{\gamma}\int_{\{1< |Du|< 2\}}\eta^2\,\,|D^{2}u|^2\,|Du|^{p-2}\,dx\nonumber\\
 &&+\Gamma^2_\gamma C_\varepsilon\int_{\{1<|Du|< 2\} }\eta^{2}\, k^{2}\,|Du|^{\gamma+p}\,dx
\end{eqnarray}
where, in the last line we used \eqref{eq48}, \eqref{phiprimo} and the fact that, since
$\dfrac{P}{\delta+P}\leqslant 1$, we have
\begin{equation}\label{eq88}
\aligned (\delta+P)^{-1}\Phi'(P)
&=\Phi^{\prime}(P)\cdot\dfrac{P}{\delta+P}\cdot P^{-1} \\
&\leqslant \Gamma_{\gamma}\Phi(P)\cdot P^{-2}=\Gamma_\gamma\,(1+P)^{\gamma-2}\\
&=\Gamma_\gamma|Du|^{\gamma-2}\qquad\qquad\text{in the
set}\,\quad\{1<|Du|<2\}.
\endaligned
\end{equation}
 Plugging \eqref{I6b} and \eqref{eq52} into \eqref{eq47.1}, we get
\begin{align*}
|I_{6}| &\leqslant
2\varepsilon\int_\Omega \eta^{2}\Phi(P)|Du|^{p-2}|D^{2}u|^{2}\,dx+\varepsilon\delta\tilde C_{\gamma}\int_{\Omega}\eta^2\,\,|D^{2}u|^2\,|Du|^{p-2}\,dx\nonumber\\
 &+\Gamma^2_\gamma C_\varepsilon\int_{\{|Du|>1\} }\eta^{2}\, k^{2}\,|Du|^{\gamma+p}\,dx.
\end{align*}
By virtue of the assumption $|D^2u|^2\,|Du|^{p-2}\in
L^1_{\mathrm{loc}}(\Omega)$, we can let $\delta\to 0$ in previous
estimate thus getting
\begin{equation}\label{52.1}
|I_{6}|\le 2\varepsilon\int_\Omega
\eta^{2}\Phi(P)|Du|^{p-2}|D^{2}u|^{2}\,dx+ C_\varepsilon
(1+\gamma)^2\int_{\{|Du|>1\} }\eta^{2}\, k^{2}\,|Du|^{\gamma+p}\,dx,
\end{equation}
where we used that $\Gamma_\gamma\sim (\gamma+1)$. For $I_{7}$, using
Young's inequality {and recalling that $P=(|Du|-1)_+$}, we get
\begin{equation}\label{eq85.2}
\aligned|I_{7}|&\le 2\int_\Omega \eta|D\eta||f||Du|\Phi(P)
=2\int_{\{|Du|>1\}} \eta|D\eta||f||Du|^{\gamma-1}(|Du|-1)^{2}\\
&\le C\int_{\{|Du|>1\}} \eta^2|f|^2(|Du|-1)^{2}|Du|^{\gamma-p}+ C\int_{\{|Du|>1\}} |D\eta|^2(|Du|-1)^{2}|Du|^{\gamma+p-2}\\
&= C\int_{\{|Du|>1\}} \eta^2|f|^2\Phi(P)|Du|^{2-p}+ C\int_{\{|Du|>1\}} |D\eta|^2\Phi(P)|Du|^{p}
\\
&\le C\int_{\{|Du|>1\}} \eta^2|f|^2\Phi(P)+ C\int_{\{|Du|>1\}} |D\eta|^2\Phi(P)|Du|^{p},
\endaligned\end{equation}
where we used that $p\ge 2$ and that the set of integration is
$\{|Du|>1\}$. Concerning $I_8$ and $I_9$, by \eqref{est2} and
Young's inequality, we have
\begin{eqnarray}
\label{eq8520} |I_{8}|+|I_{9}|&\le&
\int_\Omega\eta^2|f||D^2u|\Phi(P)+\int_{\Omega}\eta^2|f||Du||D^2u|\Phi'(P)\cr\cr
&=&
\int_{\{|Du|>1\}}\eta^2|f||D^2u|\Phi(P)+\int_{\{|Du|\ge2\}}\eta^2|f||Du||D^2u|\Phi'(P)+\int_{\{1<|Du|<2\}}\eta^2|f||Du||D^2u|\Phi'(P)\cr\cr
&\le&\int_{\{|Du|>1\}}\eta^2|f||D^2u|\Phi(P)+2\int_{\{|Du|\ge2\}}\eta^2|f||D^2u|P\Phi'(P)+\int_{\{1<|Du|<2\}}\eta^2|f||Du||D^2u|\Phi'(P)\cr\cr
&\le&
c\Gamma_\gamma\int_{\{|Du|>1\}}\eta^2|f||D^2u|\Phi(P)+\int_{\{1<|Du|<2\}}\eta^2|f||Du||D^2u|\Phi'(P).
\end{eqnarray}
 The first integral in the right hand side of \eqref{eq8520} can be estimated by Young's inequality as follows
\begin{equation}\label{eq85bis}
\aligned &c\Gamma_\gamma\int_{\{|Du|>1\}}\eta^2\,|f|\,|D^2u|\,\Phi(P)
\\
 &\le\varepsilon \int_{\{|Du|>1\}}\eta^2\,|D^2u|^2\,\Phi(P)+
 C_\varepsilon\cdot \Gamma_\gamma^2\int_{\{|Du|>1\}}\eta^2|f|^2\,\Phi(P)\\
&\leq\varepsilon\int_{\Omega}\eta^2|Du|^{p-2}|D^2u|^{2}\Phi(P)+C_\varepsilon\cdot
\Gamma_\gamma^2\int_{\Omega}\eta^2|f|^2\Phi(P),
\endaligned
\end{equation}
where we used that $|Du|^{p-2}>1$ on the set $|Du|>1$ since $p>2$,
and that, as before, $\Gamma_\gamma\sim (\gamma+1)$. To estimate the second
integral in \eqref{eq8520}, we argue as we did for $I_6$, multiplying
and dividing it by $\left(\frac{\delta+P}{\Gamma_\gamma}\right)^{\frac{1}{2}}$
with $0<\delta<1$ and  using Young's inequality. We get
\begin{eqnarray}
\label{eq85ter} &&\int_{\{1<|Du|<2\}}\eta^2|f||Du||D^2u|\Phi'(P)=
\int_{\{1<|Du|<2\}}\eta^2\Phi'(P)\left\{
\frac{\delta+P}{\Gamma_\gamma}|D^2u|^2\right\}^{\frac{1}{2}}\left\{|f|^2|Du|^2
\frac{\Gamma_\gamma}{\delta+P}\right\}^{\frac{1}{2}} \cr\cr &\le&
\frac{\varepsilon}{\Gamma_\gamma}\int_{\{1<|Du|<2\}}\eta^2(\delta+P)|D^2u|^{2}\Phi'(P)+
C_\varepsilon\cdot \Gamma_\gamma\int_{\{1<|Du|<
2\}}\eta^2\frac{|Du|^2|f|^2}{(\delta+P) }\Phi'(P)\cr\cr &\le&
\frac{\varepsilon}{\Gamma_\gamma}\int_{ \{1<|Du|<
2\}}\eta^2\,|D^2u|^{2}P\Phi'(P)+\frac{\varepsilon}{\Gamma_\gamma}
\delta\int_{\{1<|Du|< 2\}}\eta^2|D^2u|^{2}\Phi'(P)\cr\cr &&+
C_\varepsilon\cdot \Gamma_\gamma\int_{\{1<|Du|<
2\}}\eta^2\frac{|Du|^2|f|^2}{(\delta+P) }\Phi'(P)\cr\cr &\le&
\varepsilon\int_{\{1<|Du|<2\}}\eta^2|D^2u|^{2}\Phi(P)+\delta
\varepsilon \tilde C_\gamma \int_{ \{1<|Du|<
2\}}\eta^2|D^2u|^{2}|Du|^{p-2}\cr\cr &&+
C_\varepsilon\,\Gamma^2_\gamma\int_{\{1<|Du|<
2\}}\eta^2|Du|^\gamma|f|^2\cr\cr &\le&
\varepsilon\int_{\{1<|Du|<2\}}\eta^2|D^2u|^{2}|Du|^{p-2}\Phi(P)+\delta
\varepsilon \tilde C_\gamma \int_{ \{1<|Du|<
2\}}\eta^2|D^2u|^{2}|Du|^{p-2}\cr\cr &&+ C_\varepsilon
\Gamma^2_\gamma\int_{\{1<|Du|< 2\}}\eta^2|Du|^\gamma|f|^2,
\end{eqnarray}
where we used that $\delta+P\le 2P$ in the set $\{|Du|>2\}$,
inequalities \eqref{eq48}, \eqref{phiprimo}, \eqref{eq88} and that $|Du|^{p-2}>1$ in the
set $\{|Du|>1\}$. Inserting \eqref{eq85bis} and \eqref{eq85ter} in
\eqref{eq8520}  and letting $\delta\to 0$, we get
 \begin{eqnarray}
\label{eq852} |I_{8}|+|I_{9}|&\le&
 2\varepsilon \int_{\Omega}\eta^2|Du|^{p-2}|D^2u|^{2}\Phi(P)+
C_\varepsilon\, \Gamma^2_\gamma\int_{\{|Du|>
1\}}\eta^2|Du|^{\gamma}|f|^2.
\end{eqnarray}

\noindent We remind that
\begin{align}\label{eq81}
I_{2}+I_{3}=-I_{1}-I_{4}-I_{5}-I_{6}-I_{7}-I_8-I_9.
\end{align}

\noindent We now elaborate on the precise form of
$D_{\xi\xi}\F(x,\xi)$ to estimate $I_3$. To do this, we abuse of
notation and for every scalar $t$ we denote $F'(x,t)=\partial_t
F(x,t)$ and $F''(x,t)=\partial_{tt}F(x,t)$. By {\bf{(F1)}},  for
every $\xi\in\mathbb{R}^{n\times N}\setminus \{0\}$ one has
$$
D_{\xi\xi}\F(x,\xi)=\left(\frac{F''(x,|\xi|)}{|\xi|^2}-\frac{F'(x,|\xi|)}{|\xi|^3}\right)\,\xi\otimes\xi
+\frac{F'(x,|\xi|)}{|\xi|}\,\Id_{\R^{n\times N}}.
$$
Componentwise,
\begin{eqnarray*}
D_{\xi_{j}^{\beta}\xi_i^\alpha}\F(x,\xi)&=&D_{\xi_{j}^{\beta}}\left(F'(x,|\xi|)\,\frac{\xi_i^\alpha}{|\xi|}\right)\cr\cr
&=&
\left(\frac{F''(x,|\xi|)}{|\xi|^2}-\frac{F'(x,|\xi|)}{|\xi|^3}\right)\xi_{i}^\alpha\xi_j^\beta
+\frac{F'(x,|\xi|)}{|\xi|}\delta_{\xi_i^\alpha \xi_j^\beta }.
\end{eqnarray*}
Recalling the equality in \eqref{equality}, it is well known that
for a.e. $x\in \{|Du|\ge 1\},$ we have
\begin{equation}
\aligned \sum_s\langle & D_{\xi\xi}\F(x,Du)\cdot D_{x_s}Du, DP\cdot
D_{x_s}u\rangle
=\sum_{s,i,j,\alpha,\beta}D_{\xi_{j}^{\beta}\xi_i^\alpha}\F(x,Du) u_{x_s}^{\alpha}u_{x_sx_j}^{\beta} P_{x_i}\\
&=
\left(\frac{F''(x,|Du|)}{|Du|}-\frac{F'(x,|Du|)}{|Du|^2}\right)\sum_{\alpha}\left(\sum_{i}u_{x_i}^{\alpha}(|Du|)_{x_i}\right)^2+F'(x,|Du|)|D(|Du|)|^2.
\endaligned
\end{equation}
Thus,
 \begin{align*}
I_3&
=\int_{\Omega}\eta^2\Phi'(P)\frac{F''(x,|Du|)}{|Du|}\sum_{\alpha}\left(\sum_{i}
u_{x_i}^{\alpha}(|Du|)_{x_i}\right)^2\,dx
\\ &+\int_{\Omega}\eta^2\Phi'(P)\,F'(x,|Du|)\left(|D(|Du|)|^2-\frac{\sum_{\alpha}\left(\sum_{i}
u_{x_i}^{\alpha}(|Du|)_{x_i}\right)^2}{|Du|^2}\right)\,dx.
\end{align*}
The use of the Cauchy-Schwartz inequality implies \[
\sum_{\alpha}\left(\sum_{i} u_{x_i}^{\alpha}(|Du|)_{x_i}\right)^2\le
|Du|^2|D(|Du|)|^2.
\]
 and since
\begin{equation}\label{Phi'}
\Phi'(t)=(1+t)^{\gamma-3}t(\gamma t+2)
\end{equation}
is nonnegative for every $t\ge 0$ and,  by {\bf{(F2)}},
$F'(x,|Du|)\ge 0$, then we conclude that
\begin{equation}\label{I3}
I_3\ge \int_{\Omega}\eta^2
\Phi'(P)\,\frac{F''(x,|Du|)}{|Du|}\sum_{\alpha}\left(\sum_{i}
u_{x_i}^{\alpha}(|Du|)_{x_i}\right)^2\,dx\ge 0.
\end{equation}
Therefore, using that $I_{3}\geqslant 0$ together with \eqref{eq81}
we have
\begin{align}\label{eq80}
I_{2}\leqslant |I_{1}|+|I_{4}|+|I_{5}|+|I_{6}|+|I_{7}|+|I_8|+|I_9|.
\end{align}
On the other hand, the ellipticity assumption {\bf{(F2)}} gives
that
\begin{align}\label{eq82}
I_{2} &\geqslant\nu\int_\Omega
\eta^{2}\cdot\Phi(P)\cdot|Du|^{p-2}\cdot|D^{2}u|^{2}\,dx.
\end{align}
Inserting  estimates \eqref{eq46.1}, \eqref{eq47.3}, \eqref{eq47.4},
\eqref{52.1}, \eqref{eq85.2}, \eqref{eq852} and \eqref{eq82} into
\eqref{eq80}, we obtain
\begin{align}\label{eq92}
&\nu\int_\Omega \eta^{2}\cdot\Phi(P)\, |Du|^{p-2}\,|D^{2}u|^{2}\,dx\nonumber\\
&\le 6\,\varepsilon\int_\Omega \eta^{2}\Phi(P)\,
|Du|^{p-2}\,|D^{2}u|^{2}\,dx +
C_\varepsilon (L_{1})\int_{\{|Du|>1\}} |D\eta|^{2}|\, |Du|^{p+\gamma}\,dx\nonumber\\
&+C_\varepsilon\int_{\{|Du|>1\}} \eta^{2} k^{2}\,|Du|^{p+\gamma}\,dx+C_\epsilon(\gamma+1)^2\int_{\{|Du|>1\}} \eta^{2} k^{2}\,|Du|^{p+\gamma}\,dx\nonumber\\
&+C_\varepsilon(\gamma+1)^2 \int_{\{|Du|>1\}}
\eta^{2}|f|^2|Du|^{\gamma}\,dx,
\end{align}
where we used that $\Phi(P)\le (1+P)^\gamma$. We now choose
$\varepsilon=\frac{\nu}{12}$, and reabsorb the first integral in the
right hand side by the left hand side. We obtain
\begin{align*}&\int_\Omega \eta^{2}\cdot\Phi(P)\, |Du|^{p-2}\,|D^{2}u|^{2}\,dx\nonumber\\
&\le C\int_{\{|Du|>1\}} |D\eta|^{2}|\,|Du|^{p+\gamma}\,dx+C\int_{\{|Du|>1\}} \eta^{2} k^{2}\,|Du|^{p+\gamma}\,dx\nonumber\\
&+C(\gamma+1)^2\int_{\{|Du|>1\}}\eta^{2} k^{2}\,|Du|^{p+\gamma}\,dx+ C(\gamma+1)^2\int_{\{|Du|>1\}} \eta^{2}|f|^2 |Du|^{\gamma}\,dx\nonumber\\
&\le C\int_{\{|Du|>1\}} |D\eta|^{2}|\,|Du|^{p+\gamma}\,dx+C(\gamma+1)^2\int_{\{|Du|>1\}} \eta^{2} k^{2}\,|Du|^{p+\gamma}\,dx\nonumber\\
&+ C(\gamma+1)^2\int_{\{|Du|>1\}} \eta^{2}|f|^2|Du|^{\gamma}\,dx,
\end{align*}
with $C=C(n,N,p,\nu, L_1)$, that is inequality \eqref{eq70}.
%

\medskip
\noindent \textbf{\textit{Step 2.}} Fix a ball
$B_{R}(x_{0})\Subset\Omega$ and radii $0<\rho<r<t<R$. Let  $\eta\in C^{\infty}_{0}(B_{t})$ be a  cut off
function such that $\eta\equiv 1$ on $B_{r}$ and $|D\eta|\leqslant
\dfrac{C}{t-r}$. Recalling the definition of $\Phi(t)$ and that $|Du|=1+P$ on the set $\{|Du|>1\}$,  inequality \eqref{eq70}  can be written as follows
\begin{eqnarray}\label{step2.1}
&&\int_{\{|Du|>1\}}\eta^2 P^{2}\,(1+P)^{\gamma+p-4}|D^2u|^2\,dx
\cr\cr
&\le & C\int_{\{|Du|>1\}} |D\eta|^{2}(1+P)^{\gamma+p}\,dx+C(\gamma+1)^2\int_{\{|Du|>1\}} \eta^{2}(k^{2}+|f|^{2})\,(1+P)^{\gamma+p}\,dx\cr\cr
&=: &  J_1+J_2
\end{eqnarray}
 Using the properties of $\eta$, we get
\begin{align}\label{eq94}
J_{1} \leqslant
\frac{C}{(t-r)^2}\int_{B_{t} \cap \{|Du|>1\}}(1+P)^{(\gamma+p) }\,dx.
\end{align}
In order to estimate $J_2$, it is convenient to split  the  ball $B_t$ as follows
$$ B_t=\left\{x\in B_t:\, k^2+|f|^2\le \exp\left(\frac{1}{\varepsilon}\right)-e\right\}\bigcup \left\{x\in B_t:\, k^2+|f|^2> \exp\left(\frac{1}{\varepsilon}\right)-e\right\}=:B^1_\varepsilon\cup B^2_\varepsilon$$
so that
\begin{eqnarray}\label{J2}
J_2&=&C(\gamma+1)^2\left[\int_{B^1_\varepsilon\cap\{|Du|>1\}}(k^{2}+|f|^{2})\,(1+P)^{\gamma+p}\,dx+
\int_{B^2_\varepsilon\cap\{|Du|>1\}}(k^{2}+|f|^{2})\,(1+P)^{\gamma+p}\,dx\right]\cr\cr
    &\le &C(\gamma+1)^2\left[\left(\exp\left(\frac{1}{\varepsilon}\right)-e\right)\int_{B_t\cap\{|Du|>1\}}\,(1+P)^{\gamma+p}\,dx+
\int_{B^2_\varepsilon\cap\{|Du|>1\}}(k^{2}+|f|^{2})\,(1+P)^{\gamma+p}\,dx\right]\cr\cr
&\le& C(\gamma+1)^2\left[\exp\left(\frac{1}{\varepsilon}\right)\int_{B_t\cap\{|Du|>1\}}\,(1+P)^{\gamma+p}\,dx+
\int_{B^2_\varepsilon\cap\{|Du|>1\}}(k^{2}+|f|^{2})\,(1+P)^{\gamma+p}\,dx\right]\cr\cr
&=:& J_2^1+J_2^2
\end{eqnarray}
We estimate  $J_{2}^2$,  by virtue of the assumptions $k,\,f\in
L^n\log^{\alpha}L_{\rm{loc}}(\Omega)$. In fact, we have
\begin{eqnarray}\label{J22}
J_2^2&=& C(\gamma+1)^2  \int_{B^2_\varepsilon\cap\{|Du|>1\}}(k^{2}+|f|^{2})\log^{\frac{2\alpha}{n}}(e+k^{2}+|f|^{2})\log^{-\frac{2\alpha}{n}}(e+k^{2}+|f|^{2})\,(1+P)^{\gamma+p}\,dx
\cr\cr
&\le & C(\gamma+1)^2    \left(\int_{B^2_\varepsilon\cap\{|Du|>1\}}(k^{n}+|f|^{n})\log^{\alpha}(e+k^{2}+|f|^{2})\,dx\right)^{\frac{2}{n}}\cr\cr
&&\qquad\qquad\cdot\left(\int_{B^2_\varepsilon\cap\{|Du|>1\}}\log^{-\frac{2\alpha}{n-2}}(e+k^{2}+|f|^{2})\,(1+P)^{\frac{n(\gamma+p)}{n-2}}\,dx\right)^{\frac{n-2}{n}}\cr\cr
&\le& C\varepsilon^{\frac{2\alpha}{n}}(\gamma+1)^2  ||k+|f|||^\vartheta_{L^n\log^{\alpha}L}
\left(\int_{B^2_\varepsilon\cap\{|Du|>1\}}\,(1+P)^{\frac{n(\gamma+p)}{n-2}}\,dx\right)^{\frac{n-2}{n}},
\end{eqnarray}
where we used H\"older's inequality, the definition of $B_\varepsilon^2$  \eqref{normint}, {and where now $\vartheta=\vartheta(n)$}.
Inserting \eqref{J22} in \eqref{J2}, we obtain
\begin{eqnarray}\label{J2final}
J_2&\le&    C(\gamma+1)^2\exp\left(\frac{1}{\varepsilon}\right)\int_{B_t\cap\{|Du|>1\}}\,(1+P)^{\gamma+p}\,dx\cr\cr
&&\qquad+C\varepsilon^{\frac{2\alpha}{n}}(\gamma+1)^2   ||k+|f|||^\vartheta_{L^n\log^{\alpha}L}
\left(\int_{B^2_\varepsilon\cap\{|Du|>1\}}\,(1+P)^{\frac{n(\gamma+p)}{n-2}}\,dx\right)^{\frac{n-2}{n}}\end{eqnarray}
Setting
\begin{equation}\label{eq113}
E_{R}:=\norm{k}_{L^n\log ^{\alpha}L(B_R)}^{\vartheta}+\norm{f}_{L^n\log
^{\alpha}L(B_R)}^{\vartheta}
\end{equation}
and inserting \eqref{eq94} and \eqref{J2final} into
\eqref{step2.1}, and using the notation in \eqref{eq113}, we get

\begin{eqnarray}\label{eq102a}
&&\int_{\{|Du|>1\}}\eta^2\,P^2\,(1+P)^{\gamma+p-4}|D^2u|^2\,dx \le
\frac{C}{(t-r)^2}\int_{B_{t}}(1+P)^{(\gamma+p) }\,dx\cr\cr && \qquad
+ C(\gamma+1)^2\exp\left(\frac{1}{\varepsilon}\right)\int_{B_t}\,(1+P)^{\gamma+p}\,dx\cr\cr
&&\qquad+C\varepsilon^{\frac{2\alpha}{n}}(\gamma+1)^2 E_R
\left(\int_{B_t}\,(1+P)^{\frac{n(\gamma+p)}{n-2}}\,dx\right)^{\frac{n-2}{n}}.
\end{eqnarray}
{Following \cite{EMM}, for $t\ge 0$} we consider the auxiliary function
$$G(t)=1+\int_0^t (1+s)^{\frac{\gamma+p-4}{2}}s\,ds$$
{and, recall that, by Lemma 2.2 in \cite{EMM18}, the following inequalities}
\begin{equation}\label{propG}
    \frac{1}{2(\gamma+p)^2}(1+t)^{\frac{\gamma+p}{2}}\le G(t)\le 2(1+t)^{\frac{\gamma+p}{2}},\qquad\qquad\qquad G'(t)=t(1+t)^{\frac{\gamma+p-4}{2}}\,
\end{equation}
{hold true.}
Let us denote by
$$2^*=\begin{cases}\frac{2n}{n-2}\qquad\text{if}\,\,n>2\cr
\text{any exponent}\qquad \text{if}\,\,n=2.
    \end{cases}$$

\noindent We consider now the case $n>2$. By the Sobolev imbedding
Theorem, we get
$$
\left(\int_\Omega \Big| \eta
\,G(P)\Big|^{2^*}\,dx\right)^{\frac{2}{2^*}}\leq  C \int_\Omega
\Big| D(\eta \,G(P))\Big|^{2}\,dx \leq C\int_\Omega | D\eta|^2
G(P)^{2}\,dx+ C\int_{\Omega} \eta^2 G'(P)^{2}|DP|^2\,dx.
$$
Using the properties of $G(t)$ at \eqref{propG}  in the previous
inequality 
we obtain
\begin{equation}\label{apriori3d}
\aligned
\frac{1}{(\gamma+p)^4}&\left(\int_\Omega  \eta^{2^*}(1+P)^{\frac{2^*(\gamma+p)}{2} }\,dx\right)^{\frac{2}{2^*}}\\
&\leq  c\,\int_\Omega | D\eta|^2 (1+P)^{\gamma+p}\,dx +  c\int_\Omega \eta^2 (1+P)^{\gamma+p-4}\,P^2\,|DP|^2\,dx\,\\
&\leq \frac{ c}{(t-r)^2}\int_{B_t}
(1+P)^{\gamma+p}\,dx+c\int_{\Omega} \eta^2
\Phi(P)\,(1+P)^{p-2}|D^2u|^2\,dx,
\endaligned
\end{equation}
where we also used the properties of the function $\eta$. Inserting
estimate \eqref{apriori3d} in \eqref{eq102a} and using that
$\eta\equiv 1$ on $B_r$, we get
\begin{equation*}
\aligned
\left(\int_{B_r}  \,(1+P)^{\frac{n(\gamma+p)}{n-2} }\,dx\right)^{\frac{n-2}{n}}&\leq \frac{C(\gamma+p)^4}{(t-r)^2}\int_{B_{t}}(1+P)^{(\gamma+p) }\,dx+C(\gamma+p)^6\exp\left(\frac{1}{\varepsilon}\right)\int_{B_t}\,(1+P)^{\gamma+p}\,dx\\
& \qquad +C\varepsilon^{\frac{2\alpha}{n}}(\gamma+p)^6  E_R^2
\left(\int_{B_t}\,(1+P)^{\frac{n(\gamma+p)}{n-2}}\,dx\right)^{\frac{n-2}{n}},
\endaligned
\end{equation*}
with $C=C(n,N,p,\nu,L_1)$.
Choosing $\varepsilon$   such  that
$$
C(\gamma+p)^6 \,E_{R}\cdot \varepsilon^{\frac{2\alpha}{n}}=\frac 1 2\qquad\Longleftrightarrow\qquad \frac{1}{\varepsilon}=(2CE_{R})^{\frac{n}{2\alpha}}(\gamma+p)^{\frac{3n}{\alpha}} \,
$$
and setting $\Theta=2CE_R $,
we get
\begin{eqnarray*}
\left(\int_{B_r} (1+P)^{\frac{n(\gamma+p)}{n-2}
}\,dx\right)^{\frac{n-2}{n}}&\leq& \frac 1 2 \,\left(\int_{B_t}
(1+P)^{\frac{n(\gamma+p)}{n-2} }\,dx\right)^{\frac{n-2}{n}}+
\frac{C(\gamma+p)^4}{(t-r)^2}\int_{B_{R}}(1+P)^{(\gamma+p)
}\,dx\cr\cr &+ & C(\gamma+p)^6 \cdot
\exp\left(\Theta^{\frac{n}{2\alpha}}(\gamma+p)^{\frac{3n}{\alpha}}
\right)\int_{B_{R} }\,(1+P)^{(\gamma+p)}\,dx
\end{eqnarray*}
By  the iteration Lemma \ref{iter} we infer that\begin{eqnarray}\label{eq100t}
&& \left(\int_{B_{\rho}} (1+P)^{\frac{n(\gamma+p)}{n-2}
}\,dx\right)^{\frac{n-2}{n}}\leq
\frac{C(\gamma+p)^4}{(R-\rho)^2}\int_{B_{R}}(1+P)^{(\gamma+p)
}\,dx\cr\cr & &\qquad + C(\gamma+p)^6 \cdot
\exp\left(\Theta^{\frac{n}{2\alpha}}(\gamma+p)^{\frac{3n}{\alpha}}
\right)\int_{B_{R} }\,(1+P)^{(\gamma+p)}\,dx \cr\cr &\leq&
 C \frac{(\gamma+p)^6\cdot
 \exp\left(\Theta^{\frac{n}{2\alpha}}(\gamma+p)^{\frac{3n}{\alpha}} \right)}{(R-\rho)^2}
\int_{B_{R}}(1+P)^{(\gamma+p) }\,dx,
\end{eqnarray}
since we may suppose without loss of generality that $0<R-\rho<1$.
\\
\textbf{\textit{Step 3.}} Let us define the decreasing sequence of
radii $\rho_{j}$, $j\in \mathbb{N}$, by setting
\begin{align*}
\rho_{j}:=\rho+\dfrac{R-\rho}{2^{j}},
\end{align*}
and the increasing sequence of exponents
$$p_{j}= p\left(\frac{2^*}{2}\right)^{j}.$$
Since $\gamma\ge 0$ can take any value, estimate \eqref{eq100t} can
be written on every ball $B_{\rho_{j}}$ as follows
\begin{align}\label{eq102}
\bigg(\int_{B_{\rho_{j+1}}}(1+P)^{p_{j+1}}\,dx\bigg)^{\frac{1}{p_{j+1}}}
\leqslant \frac{C^{\frac{1}{p_j}}p_j^{\frac{6}{p_j}} 
\left(\exp\left(\Theta^{\frac{n}{2\alpha}}p_j^{\frac{3n}{\alpha}}
\right)\right)^{\frac{1}{p_j}}}{(\rho_{j}-\rho_{j+1})^{\frac{2}{p_j}}}
\bigg(\int_{B_{\rho_{j}}}(1+P)^{p_{j}}\,dx\bigg)^{\frac{1}{p_{j}}}.
\end{align}
Iterating estimate \eqref{eq102} we obtain
\begin{eqnarray*}
\bigg(\int_{B_{\rho}}(1+P)^{p_{m+1}}\,dx\bigg)^{\frac{1}{p_{m+1}}}
&\leqslant& \prod\limits^{m}_{j=0} \frac{C^{\frac{1}{p_j}}p_j^{\frac{6}{p_j}} 
\left(\exp\left(\Theta^{\frac{n}{2\alpha}}p_j^{\frac{3n}{\alpha}}
\right)\right)^{\frac{1}{p_j}}}{(\rho_{j}-\rho_{j+1})^{\frac{2}{p_j}}}
\bigg(\int_{B_{R}}(1+P)^{p}\,dx\bigg)^{\frac{1}{p}} \cr\cr
&\leqslant&
\prod\limits^{m}_{j=0}\frac{C^{\frac{1}{p_j}}4^{\frac{j+1}{p_j}}p_j^{\frac{6}{p_j}}
\left(\exp\left(\Theta^{\frac{n}{2\alpha}}p_j^{\frac{3n}{\alpha}}
\right)\right)^{\frac{1}{p_j}}}{(R-\rho)^{\frac{2}{p_j}}}
\bigg(\int_{B_{R}}(1+P)^{p}\,dx\bigg)^{\frac{1}{p}}
\end{eqnarray*}
where we used the definition of $\rho_j$. Now, we want to prove that the product in the right hand side of previous estimate is bounded by a constant. To this aim we write
$$\prod\limits^{m}_{j=0} \frac{C^{\frac{1}{p_j}}4^{\frac{j+1}{p_j}}p_j^{\frac{6}{p_j}} 
\left(\exp\left(\Theta^{\frac{n}{2\alpha}}p_j^{\frac{3n}{\alpha}}
\right)\right)^{\frac{1}{p_j}}}{(R-\rho)^{\frac{2}{p_j}}}=\prod\limits^{m}_{j=0}
\frac{C^{\frac{1}{p_j}}4^{\frac{j+1}{p_j}}p_j^{\frac{6}{p_j}}
}{(R-\rho)^{\frac{2}{p_j}}}\prod\limits^{m}_{j=0}
\exp\left(\Theta^{\frac{n}{2\alpha}}p_j^{\frac{3n}{\alpha}-1}
\right)$$
$$=\prod\limits^{m}_{j=0}
\frac{C^{\frac{1}{p_j}}4^{\frac{j+1}{p_j}}p_j^{\frac{6}{p_j}}
}{(R-\rho)^{\frac{2}{p_j}}}
\exp\left(\Theta^{\frac{n}{2\alpha}}\sum\limits^{m}_{j=0}
p_j^{\frac{3n}{\alpha}-1} \right)$$
$$=\displaystyle \exp\left(\log \frac{C}{(R-\rho)^2}\sum\limits^{m}_{j=0}\frac{1}{p_j} +\, 6\sum\limits^{m}_{j=0}\frac{1}{p_j}\log p_j
\, +\log 4\sum\limits^{m}_{j=0}\frac{j+1}{p_j} \,
+\Theta^{\frac{n}{2\alpha}}\sum\limits^{m}_{j=0}
p_j^{\frac{3n}{\alpha}-1}\right)
$$
$$\le  \displaystyle \exp\left(c\log \frac{C}{(R-\rho)^2}\Bigg[\sum\limits^{m}_{j=0}\frac{1}{p_j} +\, \sum\limits^{m}_{j=0}\frac{1}{p_j}\log p_j
\, +\sum\limits^{m}_{j=0}\frac{j+1}{p_j}\Bigg] \,
+\Theta^{\frac{n}{2\alpha}}\sum\limits^{m}_{j=0}
p_j^{\frac{3n}{\alpha}-1}\right).
$$
Taking in account the
definition of $p_j$, one can easily  prove that
\begin{align*}
\Bigg[\sum\limits^{m}_{j=0}\frac{1}{p_j} +\, \sum\limits^{m}_{j=0}\frac{1}{p_j}\log p_j
\, +\sum\limits^{m}_{j=0}\frac{j+1}{p_j}\Bigg]\le c(\sigma)\sum\limits^{+\infty}_{j=0}\left(\frac{2}{2^*}\right)^{j\sigma} < c(n),
\end{align*}
for some $\sigma\in (0,1)$
and
\begin{align*}
\sum\limits^{m}_{j=0} p_j^{\frac{3n}{\alpha}-1}&\leqslant
\sum\limits^{+\infty}_{j=0}
p_j^{\frac{3n}{\alpha}-1}=c(n,\alpha)
\end{align*}
since by assumption  $\frac{3n}{\alpha}-1<0$.
\\
Therefore,
\begin{align*}
\bigg(\int_{B_{\rho}}(1+P)^{p_{m+1}}\,dx\bigg)^{\frac{1}{p_{m+1}}}
&\leqslant \widehat C
\bigg(\int_{B_{R}}(1+P)^{p}\,dx\bigg)^{\frac{1}{p}},
\end{align*}
for every $m\in \mathbb{N}$. Now, letting $m\to\infty$ we end up
with
\begin{align*}
&\sup\limits_{B_{\rho}}|Du| \leq\lim_{m\to\infty}
\left(\int_{B_\rho}(1+P)^{p_m}\right)^\frac1{p_m}\leq \widehat C
\left(\int_{B_{R}}(1+|Du|)^{p}\right)^{\frac{1}{p}},\end{align*}
 that gives \eqref{apriorisup}. \\
 For further needs we observe that the constant $\widehat C$ has the form
 \begin{equation}\label{constant}
 \widehat C=c(n,p,\alpha,\rho,R)\exp\big( c(n,N,p,L_1,\nu,\alpha    )(||k||_{L^n\log^\alpha L}+||f||_{L^n\log^\alpha L})^{\vartheta}\big),
 \end{equation}
 {where $\vartheta=\vartheta(n)$.}
\\
\textbf{\textit{Step 4.}} In this Step we are going to establish
estimate \eqref{apriori3b}. To this aim it suffices to write estimate \eqref{eq70} with
$\gamma=0$ and with $\eta$ a cut off function between concentric balls $B_\rho$ and $B_R$, to get
 \begin{equation}\label{eq700}
\aligned \int_{B_\rho}\Phi(P)\,|Du|^{p-2}\,|D^2u|^2\,dx &\leq
C\int_{B_R}k^2\,|Du|^{p}\,dx
\\ &+    \frac{C}{(R-\rho)^2}\int_{B_R}\,\,|Du|^{p}\,dx
+C\int_{B_R} |f|^{2}\,dx.
\endaligned
\end{equation}
Using
estimate \eqref{apriorisup}, we obtain
\begin{equation}\label{eq701}
\aligned \int_{B_\rho}\Phi(P)\,|Du|^{p-2}|D^2u|^2\,dx
&\leq C\sup_{B_R}|Du|^p\int_{B_R}k^2\,dx +   \frac{C|B_R|}{(R-\rho)^2}\sup_{B_R}|Du|^p +C\int_{B_R}|f|^{2}\,dx \\
&\leq \hat{C} \biggm(\int_{B_{2R}}(1+|Du|)^p\,dx\bigg),
\endaligned
\end{equation}
for  a constant $\hat{C}=\hat{C}(n,N,p, L, L_1,
\nu,\rho,R,||k||_{L^n\log^\alpha L(B_{2R})},||f||_{L^n\log^\alpha
L(B_{2R})})$.
\\
The proof is finished in case $n>2$.
In the case $n=2$, it  suffices to use the appropriate
Moser-Trudinger Sobolev inequality and argue in the same way.
\end{proof}

\section{Proof of Theorem \ref{thmain1}}\label{appsubsection}

\noindent This section is devoted to the proof of Theorem \ref{thmain1}.  To this end, we state first an
approximation result, which we take from \cite{CCG, CupGuiMas}. It
shows that one can find a sequence of uniformly elliptic integrands
$\F_m$ that approximate the given $\F$. The approximants can be
chosen to be smooth in the $x$ variable, and also to have
ellipticity bounds on  the whole $\R^{n\times N}$, although
these bounds may depend on $m$ (see conditions
$\mathbf{(\tilde{F}_m2)}$-$\mathbf{(\tilde{F}_m4)}$ below).
Furthermore, these ellipticity conditions may be assumed uniform in
$m$ away from a ball of the $\xi$ variable (see conditions
$\mathbf{(\tilde{F}0)}$-$\mathbf{(\tilde{F}4)}$ below).  We recall
that, without loss of generality, we assumed that the radius $\mathcal{R}$
appearing in the assumptions {\bf{(F0)}}--{\bf{(F4)}} is equal to
$1$.

\begin{prop}\label{apprcupguimas2parte}
Let  $\F:\Omega\times\R^{n\times N}\to [0,+\infty)$ be a
Carath\'eodory function, convex  with respect to the second
variable, and  satisfying  assumptions {\bf{(F0)}}--{\bf{(F4)}}.
For a fixed   open set $\Omega'\Subset\Omega$, there exists a sequence
$\F_{m}:\Omega'\times\R^{n\times N}\to [0,+\infty)$ of
Carath\'eodory functions, $C^2$ and convex in the second variable,
such that $\F_m$ converges to $\F$ pointwise a.e. on $\Omega'$ and
everywhere in $\mathbb{R}^{{n\times}N}$. Moreover,  each $\F_{m}$
can be chosen so that the following properties are satisfied:
\begin{itemize}
\item[ $\mathbf{(\tilde{F}0)}$ ] there exist constants $\tilde L, \tilde\ell  >0$ depending on $\ell,L$ such that for all $(x,\xi)\in\Omega'\times \R^{n\times N}$
 $$\tilde\ell(|\xi|^p-1)\le \F_{m}(x,\xi)\le \tilde L(1+|\xi|)^{p},$$
\item[$\mathbf{(\tilde{F}1)}$] for every $x \in \Omega'$ and  $\xi\in \R^{nN}\setminus B_{2}(0)  $ one has $\F_{m}(x,\xi)=F_{m}(x,|\xi|)$,
\item[$\mathbf{(\tilde{F}2)}$] there exists $ \tilde\nu =\tilde\nu(\nu,p)$ such that for every  $x\in\Omega'$, $\xi\in\R^{n\times N}\setminus{B}_{2}(0)$ and  $\lambda\in \R^{n\times N}$
$$ \tilde{\nu}(1+|\xi|)^{p-2}|\lambda|^{2}\le \langle D_{\xi\xi}\F_{m}(x,\xi)\lambda,\lambda\rangle,$$
\item[$\mathbf{(\tilde{F}3)}$] there exists  $\tilde L_1>0$ such that  for every $(x,\xi)\in \Omega'\times \big(\R^{nN}\setminus B_{2}(0)\big)$
$$|D_{\xi\xi}\F_{m}(x,\xi)|\le   \tilde L_1(1+|\xi|)^{p-2},$$
\item[$\mathbf{(\tilde{F}4)}$] for every $x\in \Omega'$ and $\xi\in  \R^{nN}\setminus{B}_{2}(0)$,
$$|D_{\xi x}\F_{m}(x,\xi)|\le  {2^{p-1}}k_m(x){(1+|\xi|)^{p-1}}, $$
where $k_m\in C^{\infty}(\Omega')$ is a non-negative function such
that $k_m\to k$ strongly in $L^{n}\log^\alpha L(\Omega')$.
\end{itemize}
Moreover, the above properties can be extended to every $\xi\in
\R^{n\times N}$ in the following way:
\begin{itemize}
\item[$\mathbf{(\tilde{F}_m2)}$] There exists $\mu_{m}>0$ such that  for every $(x,\xi)\in\Omega'\times \R^{n\times N}$ and for every $\lambda\in \R^{n\times N}$
 $$\mu_{m}(1+|\xi|)^{p-2}|\lambda|^{2}
 \le \langle D_{\xi\xi}\F_{m}(x,\xi)\lambda,\lambda\rangle,$$
\item[$\mathbf{(\tilde{F}_m3)}$] there exists $\gamma_{m}>0$ such
that for every $(x,\xi)\in\Omega'\times \R^{nN}$
$$|D_{\xi\xi}\F_{m}(x,\xi)|\le  \gamma_{m}(1+|\xi|)^{p-2},$$
\item[$\mathbf{(\tilde{F}_m4)}$] there exists $\Lambda_m>0$ such that for every $x\in \Omega'$ and    $\xi\in  \overline{{B}_{2}(0)}$
$$|D_{\xi x}\F_{m}(x,\xi)|\le \Lambda_m(1+|\xi|)^{p-1}. $$
\end{itemize}
\end{prop}

\noindent We now recall a regularity result for minimizers of
functionals of the form
$$\inf_{w}\int_{\Omega}(\F(x, Dw)+\arctan(|w-\bar{u}|^2))\,dx$$
where $\F$ has standard growth conditions and smooth dependence on
the $x$-variable, and $\bar{u}$ is fixed. In absence of the
perturbation term $\arctan(|w -\bar{u} |^2)$, this regularity result
is well known. We refer to \cite{giamod86} for the   higher
differentiability result, and to \cite[Theorem 1.1]{CupGuiMas} as
far as  the Lipschitz continuity
 of the local minimizers is concerned. In presence of the perturbation term, the proofs can be easily adapted because of the boundedness of
 the function $\arctan(|w-\bar{u}|^2)$ and of its derivative with respect to the variable $w$, thus obtaining the following

\begin{theor}\label{acefus}
Let   $\G:\Omega\times\R^{n\times N}\to [0,+\infty)$, $\G\in
C^2(\Omega\times \R^{n\times N})$, and define  the functional
\begin{equation}\label{funzionale-arctg}
\int_{\Omega} \Big(\G(x, Dw)+\arctan(|w-\bar{u}|^2)\Big) \,dx
\end{equation}
with  $\bar{u}\in C^2(\Omega;\R^N)$. Assume that there exists $p\ge
2$ such that for every $x\in \Omega$ and every $\xi,\lambda \in
\R^{n\times N}$,
\[c_1|\xi|^p-c_2\le \G(x,\xi)\le L(1+|\xi|)^{p},\]
\[\nu\,(1+|\xi|)^{p-2}|\lambda|^2\le \langle D_{\xi\xi} \G(x,\xi)\lambda,\lambda\rangle,\]
 \[|D_{\xi \xi}\G(x,\xi)|\le L_1\,(1+|\xi|)^{p-2},\]
\[|D_{\xi x}\G(x,\xi)|\le K(1+|\xi|)^{p-1},\]
with positive constants $c_1,c_2,L,L_1,\nu,K$. Then any local
minimizer  $v\in W^{1,p}_{\loc}(\Omega)$ of \eqref{funzionale-arctg} is in $W_{
\mathrm{loc}}^{2,2}(\Omega;\mathbb{R}^N)$ and
\[(1+|Dv|^2)^{\frac{p-2}{2}}|D^2 v|^2\in L_{ loc}^1(\Omega).\] Moreover, if there exists  $G:\Omega\times [0,+\infty)\to [0,+\infty)$  such that $\G(x,\xi)=G(x,|\xi|)$, then one also has $v\in W^{1,\infty}_{ \mathrm{loc}}(\Omega;\R^N)$.
\end{theor}

\noindent We are now ready for proving Theorem \ref{thmain1}.

\begin{proof}[Proof of Theorem \ref{thmain1}]
Let $u\in W^{1,p}_{\loc}(\Omega)$ be a local minimizer of the
functional $\mathbb{F}(u,\Omega)$, and let $B_{r}\subset\Omega$ be a
fixed ball.  We consider the sequence of energy densities
$\F_m(x,\xi)$ obtained after applying Proposition
\ref{apprcupguimas2parte} to the integrand $\F$. For a standard
sequence of mollifiers $\rho_\varepsilon$, we set
$u_{\varepsilon}=u\ast\rho_{\varepsilon}$,
$f_{\varepsilon}=f\ast\rho_{\varepsilon}$ and define
\begin{align*}
{\mathbb F}_{\varepsilon,m}(w,B_r):=\int_{B_{r}}
\Big(\F_m(x,Dw)+f_{\varepsilon}(x)w+\arctan
|w-u_{\varepsilon}|^{2}\Big)\,dx.
\end{align*}
The lower semi-continuity and strict convexity of ${\mathbb
F}_{\varepsilon,m}$ with respect to the gradient variable ensure that the minimization problem
\begin{align}\label{eq104}
\min\bigg\{{\mathbb F}_{ \varepsilon, m  }(w;B_{r}):w\in
u+W^{1,p}_{0}(B_{r},\R^N)\bigg\}
\end{align}
has a unique solution $v_{\varepsilon,m}\in
u+W^{1,p}_{0}(B_{r},\R^N)$. By the growth conditions of  $\F_m$ stated in  $\mathbf{(\tilde{F}0)}$ of Proposition
\ref{apprcupguimas2parte} and
the minimality of $v_{\varepsilon,m}$, we deduce that $$
\aligned \tilde\ell\int_{B_{r}}|Dv_{\varepsilon,m}|^{p}\,dx
&\leq \int_{B_{r}}\big(\tilde \ell+ \F_{m} (x,Dv_{\varepsilon,m})\big)\,dx\\
&\leq \,
{\mathbb F}_{\varepsilon,m}(v_{\varepsilon,m},B_r)+\int_{B_{r}}\big(\tilde\ell-f_{\varepsilon}(x)\cdot v_{\varepsilon,m}\big)\,dx\\
&\leq {\mathbb F}_{\varepsilon,m}(u, B_r)+\int_{B_{r}}\Big(\tilde	\ell-f_{\varepsilon}(x)\cdot v_{\varepsilon,m}\big)\,dx\\
&= \int_{B_{r}}\bigg( \F_{m} (x,Du)+\arctan |u-u_{\varepsilon}|^{2}\bigg)\,dx-\int_{B_{r}}f_{\varepsilon}(x)(v_{\varepsilon,m}-u)\,dx +\tilde\ell|B_r|\\
&\leq
\int_{B_{r}}\bigg(\tilde L|Du|^{p}+\dfrac{\pi}{2}\bigg)\,dx+\int_{B_{r}}|f_{\varepsilon}(x)||v_{\varepsilon,m}-u|\,dx+(\tilde\ell+\tilde L)|B_r|.
\endaligned
$$
Hence
$$\int_{B_{r}}|Dv_{\varepsilon,m}|^{p}\,dx
\leq C\int_{B_{r}}|Du|^{p}\,dx+\int_{B_{r}}|f_{\varepsilon}(x)||v_{\varepsilon,m}-u|\,dx+C|B_r| $$
with a constant $C=C(\tilde \ell,\tilde L)$
independent of $m$ and  of $\varepsilon$.
\\
Using  Young's   and Poincar\'e inequalities in the last integral of previous estimate, we get
$$
\aligned \int_{B_{r}}|&Dv_{\varepsilon,m}|^{p}\,dx
\leq C\int_{B_{r}}\big(1+|Du|^{p}\big)\,dx+c_{\kappa}\int_{B_{r}}|f_{\varepsilon}(x)|^{p^{\prime}}\,dx+\kappa\int_{B_{r}}|v_{\varepsilon,m}-u|^{p}\,dx\\
&\leq C\int_{B_{r}} \big(1+|Du|^{p}\big)\,dx+ c_{\kappa} \int_{B_{r}}|f_{\varepsilon}(x)|^{p^{\prime}}\,dx+ \kappa \,c_{n,p,r} \int_{B_{r}}|Dv_{\varepsilon,m}-Du|^{p}\,dx,\\
&\leq
C\int_{B_{r}}\big(1+|Du|^{p}\big)\,dx+c_{\kappa}\int_{B_{r}}|f_{\varepsilon}(x)|^{p^{\prime}}\,dx+\kappa\,
c_{n,p,r}\int_{B_{r}}|Du|^{p}\,dx+ \kappa\,
c_{n,p,r}\int_{B_{r}}|Dv_{\varepsilon,m}|^{p}\,dx,
\endaligned
$$
Choosing $\kappa=\frac{1}{2 c_{n,p,r}}$, we can reabsorb the last
integral in the right hand side of previous inequality by the left
hand side, getting
$$\int_{B_{r}}|Dv_{\varepsilon,m}|^{p}\,dx\le C\int_{B_{r}}\big(1+|Du|^{p}\big)\,dx
+C\int_{B_{r}}|f_{\varepsilon}(x)|^{p^{\prime}}\,dx$$ for a constant
$C$ independent of $\varepsilon$ and $m$ and so with right hand side
independent of $m$. 
Moreover, since $p^{\prime}\le 2\le n$, the strong convergence of $f_\varepsilon$
to  $f$  in $L^n\log^\alpha L$, implies that
$$\int_{B_{r}}|f_{\varepsilon}(x)|^{p'}\,dx\leqslant C,$$ with $C$
independent of $\varepsilon$. Therefore
\begin{equation}\label{boun}
	\int_{B_{r}}|Dv_{\varepsilon,m}|^{p}\,dx\le C
\end{equation}
with a constant
$C$ independent of $\varepsilon$ and $m$.
Hence, by weak compactness, we deduce that there exists $v_{\varepsilon}\in u+W^{1,p}_{0}(B_{r};\R^N)$  such that 
\begin{equation}\label{weak1}
\{v_{\varepsilon,m}\}_m\rightharpoonup v_\varepsilon\qquad\qquad\text{weakly  in }\,\,W^{1,p}(B_r)	
\end{equation} 
and
\begin{equation}\label{strong1}
\{v_{\varepsilon,m}\}_m\to v_\varepsilon\qquad\qquad\text{strongly  in }\,\, L^{p}(B_r)
\end{equation} 
as $m\rightarrow
+\infty$, up to a subsequence. Set now
\begin{align*}
{\mathbb F}_{\varepsilon}(w,B_r):=\int_{B_r}  \Big(\F
(x,Dw)+f_{\varepsilon}(x)w+\arctan (|w-u_{\varepsilon}|^{2})\Big)\,\,dx.
\end{align*}
For every fixed $\varepsilon>0$, one can see that the functionals ${\mathbb
F}_{\varepsilon,m}$ $\Gamma$-converge to ${\mathbb F}_{\varepsilon}$ as
$m\to\infty$ (see Theorem 5.14 and Corollary 7.20 in
\cite{DM:GammaConv}). As a consequence, $v_{\varepsilon}$ is a
minimizer of ${\mathbb F}_{\varepsilon}$. Now, the weak lower semicontinuity
of the $L^{p}$  norm and \eqref{boun} imply that
\begin{equation}\label{bound1}
\int_{B_{r}}|Dv_{\varepsilon}|^{p}\,dx \le \liminf_{m\to\infty}\int_{B_{r}}|Dv_{\varepsilon,m}|^{p}\,dx \leqslant C.
\end{equation}
Since the constant $C$ in \eqref{bound1} is independent of $\varepsilon$, as before, by compactness there exists $v\in
u+W^{1,p}_{0}(B_{r};\R^N)$ such that 
\begin{equation}\label{weak2}
\{v_{\varepsilon}\}_\varepsilon\rightharpoonup v\qquad\qquad\text{weakly  in }\,\,W^{1,p}(B_r)	\end{equation} 
and
\begin{equation}\label{strong2}
\{v_{\varepsilon}\}_\varepsilon\to v\qquad\qquad\text{strongly  in }\,\, L^{p}(B_r)
\end{equation} 
as $\varepsilon\to 0$, up to a subsequence. Also, again by the weak lower
semicontinuity of the norm,
\begin{equation}\label{bound2}
\int_{B_{r}}|Dv|^{p}\,dx\le\liminf\limits_{\varepsilon\rightarrow
0}\int_{B_{r}}|Dv_{\varepsilon}|^{p}\,dx \leqslant C.
\end{equation}
Observe that, as $\varepsilon\to 0$, the functionals ${\mathbb
F}_{\varepsilon}$ $\Gamma$-converge to
\begin{align*}
{\mathbb F}_0(w,B_r):=\int_{B_r}\Big(\F(x,Dw)+f(x)\cdot w+\arctan
(|w-u|^{2})\Big)\,dx,
\end{align*}
whence $v$ is a minimizer of ${\mathbb F}_0$ coinciding with $u$ on $\partial B_r$, and therefore
${\mathbb F}_0(v,B_r)\leq {\mathbb F}_0(u,B_r)$. This, together with
the minimality of $u$, implies that
\begin{equation*}
{\mathbb F}(u,B_r) \leq {\mathbb F}(v, B_r) \leq {\mathbb F}_0(v,
B_r)\leq {\mathbb F}_0(u,B_r)={\mathbb F}(u, B_r).\end{equation*}
Hence all the above inequalities hold true as equalities, and as a consequence
$$
\int_{B_r}\arctan |u-v|^{2}\,dx=0,\qquad\Rightarrow \qquad u=v\qquad
\text{a.e. in }B_r.
$$
Since the functionals ${\mathbb F}_{\varepsilon,m}$ satisfy the
assumptions of Theorem \ref{acefus} for every $\varepsilon$ and $m$,
their minimizers $v_{\varepsilon,m}$ belong to
$W^{2,2}_{\loc}\cap W^{1,\infty}$. Therefore we are legitimate to use
the a priori estimate at \eqref{apr} of  Theorem
\ref{t:Apriori2}, thus getting
\begin{align}
\label{conv} \sup\limits_{B_{\rho}}|Dv_{\varepsilon,m}|\leqslant
C\exp\big( C(||k_\varepsilon||_{L^n\log^\alpha L}+||f_\varepsilon||_{L^n\log^\alpha L})^{\vartheta}\big)\bigg(\int_{B_{
\rho' }}(1+  |Dv_{\varepsilon,m}|^p) \,dx\bigg)^{\frac{1}{p}},
\end{align}
for every $B_\rho\subset B_{\rho'}\Subset B_r$ and for a constant
$C$ independent of $\varepsilon$ and $m$ by the expression at \eqref{constant}. By  virtue of
\eqref{bound1} and \eqref{bound2}, and since $k_\varepsilon \to k$ and
$f_\varepsilon\to f$ strongly in $L^n\log^\alpha L$, passing to the limit, first
as $m\rightarrow +\infty$, and then as $\varepsilon\rightarrow 0$ in
estimate \eqref{conv}, we conclude that
\begin{align*}
\sup\limits_{B_{\rho}}|Du|\leqslant
C\exp\big( C(||k||_{L^n\log^\alpha L}+||f||_{L^n\log^\alpha L})^{\vartheta}\big)\bigg(\int_{B_{
\rho'}}(1+|Du|^{p})\,dx\bigg).
\end{align*}
This finishes the proof.
\end{proof}

\section{Proof of Theorem \ref{thmain2}}\label{ssobolev}

\noindent In this section, we establish the integrability of second
order distributional derivatives of the local minimizers of the
functional $\mathbb{F}(\cdot,\Omega)$.  To this aim, we will need a measure theory tool, that is inspired by a result in \cite{CGGP}.
The result we present here goes a bit further, as it states
convergence on the set $\{|Df|>1\}$, and not only on $\{|Df|>t\}$
for each $t>1$ and this improvement is due to the Lipschitz
regularity of the minimizers proven in Theorem \ref{thmain1}.  The precise statement is
as follows.

\begin{prop}
        \label{anto}
Let  $p\ge 2$, $N\ge 1$, and let $f_k, f\in
W^{1,p}(\Omega;\mathbb{R}^N)$ be given, and denote
$P_k=(|Df_k|-1)_+$. Assume that:
\begin{itemize}
\item[(a)] $f_k\rightharpoonup f$ in $W^{1,p}(\Omega;\mathbb{R}^N)$,
\item[(b)] $P_k\in L^\infty$ and there exists a positive constant $M$ independent of $k$ such that
\begin{equation}\label{uniformly}\|P_k\|_{L^{\infty}(\Omega)}\leq M\end{equation}
for every  $k\in\mathbb{N}$ .
\item[(c)]
 Assume that there exists a positive constant $N$ independent of $k$ such that
 \begin{equation}\label{0.1}\int_{\Omega} \frac{P_k^p}{(1+P_k)^2}\,|DP_k|^2\,dx\leq N\end{equation}
 for every  $k\in\mathbb{N}$.\end{itemize}
Then one has $(|Df|-1)_+^{\frac{p}{2}+1}\in  W^{1,2}(\Omega)$.
Moreover, there exists a not relabeled  subsequence $f_{k}$ such
that
\[\displaystyle \lim_{k\to \infty}|Df_{k}|=  |Df| \quad \text{strongly in $L^{p+2}\Big(\Omega\cap \{|Df|>1\}\Big)$  },\]
and \[\displaystyle \lim_{k\to \infty}|Df_{k}|=  |Df|\quad
\text{a.e. in}\ \Omega\cap \{|Df|> 1\}.\]
\end{prop}
\begin{proof}
First, it is immediate to see that
$$\aligned
\int_\Omega \left|D\left(P_k^{\frac{p+2}2}\right)\right|^2
&=c(p)\int_\Omega P_k^p \, |D P_k|^2=c(p)\int_\Omega \frac{P_k^p}{(1+P_k)^2} \, (1+P_k)^2|D P_k|^2\\
&\leq c(p,M)\int_\Omega \frac{P_k^p}{(1+P_k)^2}\,|DP_k|^2\leq c(p,M,N),
\endaligned$$
where we used assumptions $(b)$ and $(c)$.
By compactness, there exists $\varphi\in W^{1,2}(\Omega)$ such that
$P_k^{\frac{p+2}2}\rightharpoonup \varphi$ weakly in $W^{1,2}(\Omega)$ and  strongly  in $L^2(\Omega)$. As a consequence,
$\varphi\geq 0$ almost everywhere. Using also that $r\mapsto
r^\frac{2}{2+p}$ is $\frac{2}{2+p}$-H\"older continuous on
$[0,\infty)$,  we can deduce that
$$
\int_\Omega \Big|P_k-\varphi^\frac2{p+2}\Big|^{p+2}\leq c(p)\int_\Omega
\Big|P_k^\frac{p+2}{2}-\varphi\Big|^2
$$
and therefore $P_k\to \varphi^\frac2{p+2}$ strongly in $L^{p+2}(\Omega)$  and  also in measure.
From now on, let us denote $P=\varphi^\frac{2}{p+2}$ 
and, recalling the definition of $P_k$, observe that
\begin{eqnarray*}
&&\lim_{k\to+\infty}\int_\Omega\big|P_k-P\big|^{p+2}\cr\cr
&=&\lim_{k\to+\infty}\int_{\{x\in\Omega:\,|Df_k|\le 1\}}|P|^{p+2}+
\lim_{k\to+\infty}\int_{\{x\in\Omega:\,|Df_k|> 1\}}\big||Df_k|-1-P\big|^{p+2}=0
\end{eqnarray*}
that, in particular,  implies
\begin{equation}\label{strong2bis}
\lim_{k\to+\infty}\int_{\{x\in\Omega:\,|Df_k|> 1\}}\big||Df_k|-1-P\big|^{p+2}=0.
\end{equation}
Moreover by the convergence in measure of $P_k$ to $P$, for every
$t>0$, we have
\begin{eqnarray*}
0&=&\lim_{k\to+\infty}|\{x\in \Omega    :\,\,
\big|\big(|Df_k|-1\big)_+-P|\ge t\}\big|\cr\cr &=&
\lim_{k\to+\infty}|\{x\in \Omega    :\,\, |Df_k|\le 1\,\,
\text{and}\,\,\big|P|\ge t\}\big|\cr\cr &+&\lim_{k\to+\infty}|\{x\in
\Omega :\,\, |Df_k|> 1\,\, \text{and}\,\,\big||Df_k|-1-P|\ge
t\}\big|,
\end{eqnarray*}
that, in particular, yields
\begin{eqnarray}\label{measure1}
0&=&\lim_{k\to+\infty}|\{x\in \Omega    :\,\, |Df_k|\le 1\,\,
\text{and}\,\,\big|P|\ge t\}\big|.
\end{eqnarray}
For $t>0$, we write
\begin{eqnarray}\label{split}
&&\int_{\{x\in\Omega:\,P\ge t\}}\big||Df_k|-1-P\big|^{p+2}\cr\cr
&=&\int_{\{|Df_k|\le 1\,\, \text{and}\,\,P\ge
t\}}\big||Df_k|-1-P\big|^{p+2}    +\int_{\{|Df_k|> 1\,\,
\text{and}\,\,P\ge t\}}\big||Df_k|-1-P\big|^{p+2}\cr\cr &=&
I_{1,k}+I_{2,k}.
\end{eqnarray}
By \eqref{strong2bis} we have that
\begin{eqnarray}\label{I2}
\lim_{k\to+\infty}I_{2,k} &\le &\lim_{k\to+\infty}\int_{\{|Df_k|>
1\}}\Big||Df_k|-1-P\Big|^{p+2}=0.
    \end{eqnarray}
Since by virtue of the assumption $(b)$, we have $||P||_\infty\le
c(M)$, we get
\begin{eqnarray}\label{I1}
\lim_{k\to+\infty}I_{1,k}&\le
&c(p)\lim_{k\to+\infty}\int_{\{|Df_k|\le 1\,\, \text{and}\,\,P\ge
t\}}(2+|P|)^{p+2}\cr\cr &\le &c(M,p)\lim_{k\to+\infty}|\{|Df_k|\le
1\,\, \text{and}\,\,P\ge t\}|=0, \end{eqnarray} by the equality in
\eqref{measure1}. Inserting \eqref{I1} and \eqref{I2} in
\eqref{split}, we conclude
\begin{equation}\label{strongGG}\lim_{k\to +\infty}\int_{\{P\ge t\}}\Big||Df_k|-1-P\Big|^{p+2}=0\end{equation}
for every $t>0$ .
 One can easily observe that
 $$B(0):=\{x\in \Omega\,:\,P>0\}=\bigcup_{n\in\mathbb{N}}\left\{x\in \Omega\,:\, P\ge\frac{1}{n}\right\}=:\bigcup_{n\in\mathbb{N}}B\left(\frac{1}{n}\right)$$
 and, since $B\left(\frac{1}{n}\right)\subset B\left(\frac{1}{n+1}\right)$,
 $$|\{x\in \Omega\,:\, P>0\}|=\lim_{n\to+\infty}\left|\left\{x\in \Omega\,:\, P\ge \frac{1}{n}\right\}\right|.$$
 Since $B\left(\frac{1}{n}\right)\subset B(0)$, for every $n\in \mathbb{N}$,
\[\|\chi_{B(0)}-\chi_{B(1/n)}\|_{L^{1}(\Omega)}=\|\chi_{B(0)\setminus B(1/n)}\|_{L^{1}(\Omega)}=|B(0)- B(1/n)|\to 0. \]
 Therefore
\begin{eqnarray}\label{split1}
&&\int_{\{P>0\}}\big||Df_k|-1-P\big|^{p+2}\cr\cr
&=&\int_{\{P>0\}}\big||Df_k|-1-P\big|^{p+2} -\int_{\{P\ge
1/n\}}\big||Df_k|-1-P\big|^{p+2}\cr\cr &+&\int_{\{P\ge
1/n\}}\big||Df_k|-1-P\big|^{p+2}\cr\cr
&=&\int\big||Df_k|-1-P\big|^{p+2}(\chi_{B(0)}-\chi_{B(1/n)}) \cr\cr
&+&\int_{\{P\ge 1/n\}}\big||Df_k|-1-P\big|^{p+2}\cr\cr &\le &
C(M)\|\chi_{B(0)}-\chi_{B(1/n)}\|_{L^{1}(\Omega)}+\int_{\{P\ge
1/n\}}\big||Df_k|-1-P\big|^{p+2},
\end{eqnarray}
where we used assumption $(b)$.
Passing to the limit first as $k\to \infty$ and using
\eqref{strongGG} with $t=1/n$, and then as $n\to \infty$ in
\eqref{split1}, we conclude that
\begin{eqnarray}
\lim_{k\to \infty}\int_{\{P>0\}}\big||Df_k|-1-P\big|^{p+2}=0.
\end{eqnarray}
From previous equality we deduce that
 \begin{equation}\label{stronG}
|Df_k|\to P+1\qquad\text{strongly in }\,\,\, L^{p+2}({\{P>0\}})
\end{equation}
and, of course, modulo subsequences,  also  weakly and  almost
everywhere. By assumption,  $(f_k)$ is  weakly convergent in
$W^{1,p}(\Omega;\mathbb{R}^N)$ to  $f$, so, by the essential
uniqueness of the weak limit
 we conclude that
\begin{equation}\label{strongGGG}|Df|=P+1\qquad\qquad \text{a.e.\,\,in}\,\, {\{P>0\}}.\end{equation}
Therefore \[{\{P>0\}}=\{x\in \Omega\,:\, |Df|>1\}.\] By
\eqref{stronG} and \eqref{strongGGG} and  the equality above,
there exists a subsequence (not relabeled) of $f_k$ such that
\begin{equation}\label{strongG4}|Df_k|\displaystyle \to_{k\to \infty} |Df| \qquad \text{strongly in $L^{p+2}\Big(\{x\in \Omega\,:\, |Df|>1\}\Big)$}.
 \end{equation}
 The fact that $(|Df|-1)_+^{\frac{p}{2}+1}\in  W^{1,2}(\Omega)$ now easily follows by the semicontinuity of the $L^2$ norm, assumption \eqref{0.1} and \eqref{strongGGG}.
 
%
%
%
%
%
\end{proof}

\noindent We can now proceed with the proof of Theorem \ref{thmain2}.

\begin{proof}[Proof of theorem \ref{thmain2}]
Let $v_{\varepsilon,m}$ be the solution of the problem
\eqref{eq104}. The functionals $\mathbb{F}_{\varepsilon,m}$,  for
every $\varepsilon$ and $m$, satisfy the assumptions of Theorem
\ref{acefus}. Therefore their minimizers $v_{\varepsilon,m}$ belong
to $W^{2,2}_{\loc}\bigcap W^{1,\infty}$ and are such that
$|Dv_{\epsilon,m}|^{p-2}|D^2v_{\varepsilon,m}|^2\in  L^1$. As a
consequence, we are legitimate to use the a priori estimate
\eqref{apriori3} of  Theorem \ref{t:Apriori2}, thus getting
\begin{equation}\label{342}
\int_{B_{\rho}(x_0)}
 \frac{P_{\varepsilon,m}^2}{(1+P_{\varepsilon,m})^2}\,|Dv_{\varepsilon,m}|^{p-2}|D^2v_{\varepsilon,m}|^2\,dx\leq C
     \int_{B_{2R}(x_0)}\left(1+|Dv_{\varepsilon,m}|^p\right)\,\,dx,
\end{equation}
with  constant $\hat C=C\exp(C(\|k_\varepsilon\|_{L^{n}\log^\alpha L_{\loc}(\Omega)}+\|f_\varepsilon\|_{L^{n}\log^\alpha L_{\loc}(\Omega)})^\vartheta)$ independent of  $m$. As
usual, $P_{\varepsilon,m}=(|Dv_{\varepsilon,m}|-1)_+$. Now combining
estimates   \eqref{342} and \eqref{boun}, we obtain
\begin{equation}\label{weak1}
\int_{B_{\rho}(x_0)}|D(\mathcal{G}(P_{\varepsilon,m}
))|^2\,dx=\int_{B_{\rho}(x_0)}
\frac{P_{\varepsilon,m}^2}{(1+P_{\varepsilon,m})^2}\,|Dv_{\varepsilon,m}|^{p-2}|D^2v_{\varepsilon,m}|^2\,\dt
x\le C
\end{equation}
with a constant $C$ independent of  $m$ and of $\varepsilon$, and where $\mathcal{G}$ is the function defined in \eqref{funzioneg}.
Furthermore, by estimates \eqref{apr} and \eqref{boun},  we have that
$$
\|P_{\varepsilon,m}\|_{L^\infty(B_\rho)}\leq\sup_{B_\rho}|Dv_{\varepsilon,m}|\le
C,$$ with a constant $C$ independent of $m$ and $\varepsilon$. Thus, recalling that $v_{\varepsilon,m}$ weakly converges to $v_\varepsilon$ in $W^{1,p}(B_r)$, we are legitimate to apply
Proposition \ref{anto} to the sequence $(P_{\varepsilon,m})_m$ and  this yields
\begin{equation}\label{strong1}
|Dv_{\varepsilon,m}|\to |Dv_\varepsilon|\qquad \text{strongly in }\,\,
L^p(B_\rho\cap \{|Dv_\varepsilon|>1\}),
\end{equation}
and
\begin{equation}\label{strong2}
|Dv_{\varepsilon,m}|\to |Dv_\varepsilon| \qquad \text{a.e.
in}\,\,B_\rho\cap \{|Dv_\varepsilon|>1\},
\end{equation}
as $m\to\infty$.
If we now set $w_{\varepsilon,m}=\mathcal{G}(P_{\varepsilon,m})$,
then from \eqref{weak1} we deduce that (up to a subsequence) one has
$w_{\varepsilon,m}\to w_\varepsilon$ as $m\to \infty$ for some
$w_\varepsilon\in W^{1,2}(B_\rho )$, with weak convergence in
$W^{1,2}(B_\rho )$, strong convergence in $L^{2}(B_\rho)$, and a.e.
convergence in $B_\rho$. The latter, together with \eqref{strong2},
implies that $w_\varepsilon=\mathcal{G}((|Dv_\varepsilon|-1)_+)$
almost everywhere on $B_\rho\cap \{|Dv_\varepsilon|>1\}$. But then
the lower semicontinuity of the norm implies that
\begin{equation}\label{weak4}
\int_{B_{\rho}(x_0)}|D(\mathcal{G}((|Dv_{\varepsilon}|-1)_+))|^2\le
C,
\end{equation}
with a constant $C$ independent of $\varepsilon$. We now argue for
the sequence $v_\varepsilon$ as we did for $v_{\varepsilon,m}$, and
the theorem follows .
\end{proof}

\vspace{0.5cm}  \emph{Acknowledgements}.
A. Clop and F. Hatami were partially supported by project FP7-607647 of the European Comission. A. Clop is also supported by projects 2017-SGR-395 (Catalan Government) and MTM2016-75390 (Spanish Government).
R.Giova and A.Passarelli di
Napoli have been partially supported by the Gruppo Nazionale per
l'Analisi Matematica, la Probabilit\`{a} e le loro
 Applicazioni (GNAMPA) of the Istituto Nazionale di Alta Matematica
 (INdAM). \\
 \noindent R. Giova has been partially supported by
 Universit\`{a} degli Studi di Napoli Parthenope through the projects
\lq\lq sostegno alla Ricerca individuale\rq\rq (2015 - 2016 - 2017)
and ``Sostenibilit\`a, esternalit\`a e uso efficiente delle risorse
 ambientali''(2017-2019)

\medskip

\noindent {\bf A. Clop, F. Hatami}\\
Departament de Matem\`atiques,\\
Facultat de Ci\`encies, Campus de la U.A.B.\\
08193-Bellaterra (CATALONIA)\\
\noindent {\em E-mail address}: albertcp@mat.uab.cat,
fhatami@mat.uab.cat
\medskip

\noindent {\bf R. Giova}\\
\noindent Universit\`{a} degli Studi di Napoli ``Parthenope'' \\
Palazzo Pacanowsky - Via Generale Parisi, 13 \\
 80132 Napoli, Italy

\noindent {\em E-mail address}: raffaella.giova@uniparthenope.it

\medskip

\noindent
{\bf A. Passarelli di Napoli}\\
\noindent Universit\`{a} degli Studi di Napoli ``Federico II''\\
Dipartimento di Mat.~e Appl. ``R.~Caccioppoli''\\
 Via Cintia
  80126
Napoli, Italy\\
\noindent {\em E-mail address}: antpassa@unina.it
\medskip

\begin{thebibliography}{bib}
\bibitem{Adams} R.A. Adams, \emph{ Sobolev spaces.} Academic Press, New York-London 1975, Pure and Applied Mathematics, vol. 65.
\bibitem{B}  L. Brasco,\emph{ Global $L^\infty$ gradient estimates for solutions to a certain degenerate elliptic equation}. Nonlinear Anal. 74 (2011), no. 2, 516--531.
\bibitem{BCS} L. Brasco, G. Carlier, F. Santambrogio, \emph{Congested traffic dynamics, weak flows and very degenerate elliptic equations. } J. Math. Pures Appl. (9) 93 (2010), no. 6, 652--671.
\bibitem{CJS} G. Carlier, C. Jimenez, F. Santambrogio, \emph{ Optimal transportation with traffic congestion and Wardrop equilibria}. SIAM J. Control Optim. 47 (2008), no. 3, 1330--1350.
\bibitem{CCG} P. Celada, G. Cupini, M. Guidorzi, \emph{Existence and regularity of minimizers of nonconvex integrals with $p$-$q$ growth.}
\ ESAIM Control Optim. Calc. Var. {13}{\  (2007), 343-358}.
\bibitem{CE}
\newblock M.~Chipot and  L.C.~Evans, 
\newblock Linearization at infinity and Lipschitz estimates for certain problems in Calculus of Variations. 
\newblock {\em Proc. Roy. Soc. Edinburgh Sect. A} 102 (1986), 291--303.
\bibitem{Cianchi}
\newblock A. Cianchi, 
\newblock Interpolation of operators and Sobolev embedding theorem in Orlicz spaces. \newblock \emph{ International Conference on Differential Equations (Lisboa, 1995)}. World Sci. Publ., River Edge, pp. 306-310 (1998)
\bibitem{CFMOZ} A. Clop, D. Faraco, J. Mateu, J. Orobitx, X. Zhong, \emph{Beltrami equations with coefficient in the Sobolev space $W^{1,p}$.} Publicacions Matemàtiques 53 (1), 197--230.
\bibitem{CGGP} G. Cupini, F. Giannetti, R. Giova, A. Passarelli di Napoli,  \emph{Higher integrability for minimizers of asymptotically convex integrals with discontinuous coefficients.} Nonlinear Anal. 154 (2017),  7 - 24
\bibitem{CGGP2} G. Cupini, F. Giannetti, R. Giova, A. Passarelli di Napoli,  \emph{Regularity results for vectorial minimizers of a class of degenerate convex integrals.} J.  Differential Equations 265 (2018), no.9, 4375--4416.
\bibitem{CupGuiMas} G. Cupini, M. Guidorzi, E. Mascolo, \emph{Regularity of minimizers of vectorial integrals with p-q growth}. Nonlinear Anal. 54 (2003), no. 4, 591--616.
\bibitem{DM:GammaConv}G.~Dal~Maso, \emph{An introduction to {$\Gamma$}-convergence.} Progr. Nonlinear Differential Equations Appl., no.~8, Birkh{\"{a}}user, Boston, 1993.
\bibitem{EMM} M. Eleuteri, P. Marcellini, E. Mascolo, \emph{Lipschitz estimates for systems with ellipticity conditions at infinity.} Ann. Mat. Pura Appl., 195 (2016), 1575--1603.
\bibitem{EMM16} M. Eleuteri, P. Marcellini, E. Mascolo, \emph{Lipschitz continuity for energy integrals with variable exponents.}, Rend. Lincei, Matematica E Applicazioni, 27 (2016), no. 1, p. 61
\bibitem{EMM18} M. Eleuteri, P. Marcellini, E. Mascolo, \emph{Regularity for scalar integrals without structure conditions.} Adv. Calc. Var.,  DOI: doi.org/10.1515/acv-2017-0037
\bibitem{FFM} I. Fonseca, N.Fusco, P. Marcellini, \emph{An existence result for a nonconvex variational problem via regularity}. ESAIM Control Optim. Calc. Var. 7 (2002), 69--95.
\bibitem{fossnapoliverde1}
\newblock M.~Foss, A.~Passarelli di Napoli and A.~Verde,
\newblock \emph{Global Morrey regularity results for asymptotically convex variational problems.} \newblock Forum Math. 20 (2008), 921--953.

\bibitem{fossnapoliverde2}
\newblock M.~Foss, A.~Passarelli di Napoli and A.~Verde,
\newblock  \emph{Global Lipschitz regularity for almost minimizers of asymptotically convex variational problems.}
\newblock  Ann. Mat. Pura Appl.  189 (2010), no.4, 127--162. 

\bibitem{fossgoodrich1} M. Foss, C. S. Goodrich, \emph{Partial H\"{o}lder continuity of minimizers of functionals satisfying a general asymptotic relatedness condition.} J. Convex Anal. 22 (2015), 219--246.

\bibitem{fossgoodrich2} M. Foss, C. S. Goodrich, \emph{On partial H\"{o}lder continuity and a Caccioppoli inequality for minimizers of asymptotically convex functionals between Riemannian manifolds.} Ann. Mat. Pura Appl.  195 (2016),no. 4, 1405--1461.




\bibitem{goodrich4} C. S. Goodrich, \emph{On nonlinear boundary conditions satisfying certain asymptotic behavior.} Nonlinear Anal. 76 (2013), 58--67.
%
\bibitem{goodrich1} C. S. Goodrich, \emph{Partial H\"{o}lder continuity of minimizers of functionals satisfying a VMO condition.} Adv. Calc. Var. 10 (2017), 83--110.
%
\bibitem{goodrich2} C. S. Goodrich, \emph{Partial Lipschitz regularity of minimizers of asymptotically convex functionals with general growth conditions}. J. Differential Equations 263 (2017), 4400--4428.
%
\bibitem{goodrich3} C. S. Goodrich, \emph{Partial regularity of minimizers of functionals with discontinuous coefficients of low integrability with applications to nonlinear elliptic systems}. Comm. Partial Differential Equations, doi: 10.1080/03605302.2018.1517794




\bibitem{giamod86} M. Giaquinta, G. Modica, \emph{Remarks on the regularity of the minimizers of certain degenerate
functionals}. Manuscripta Math. 57 (1986) 55-99.

 \bibitem{Gio} R. Giova, \emph{Higher differentiability for $n$-harmonic systems with Sobolev coefficients.} { J. Differential Equations} 259 (2015), 5667-5687.
 
\bibitem{Giova2}R. Giova, {\em  Regularity results for non-autonomous functionals with $L\log {L}$ -growth and Orlicz Sobolev
coefficients}.
 NoDEA Nonlinear Differential Equations Appl.  23  (2016),  no. 6, Art. 64, 18 pp.
 
\bibitem{GP2016}   R. Giova, A.~Passarelli di Napoli,  \emph{Higher differentiability of
a priori bounded  minimizers of convex variational integrals with
discontinuous coefficients} - Advances in Calculus of Variations 12 (2019), no. 1, 85--110. 

\bibitem{gi}{E.~Giusti}.
{\em Direct methods in the calculus of variations}. World
Scientific, 2003.

\bibitem{Kronz}
\newblock M.~Kronz,
\newblock {\em Partial regularity results for minimizers of quasiconvex functionals of higher order.} 
\newblock  Ann. Inst. H. Poincaré Anal. Non Lin\'eaire 19 (2002), no. 1, 81--112.

\bibitem{LPV}
\newblock C.~Leone, A.~Passarelli di Napoli and  A.~Verde, 
\newblock {\em Lipschitz regularity for some asymptotically subquadratic problems.}
\newblock  Nonlinear Anal. 67 (2007), 1532--1539.

\bibitem{APdN1}{A.~Passarelli di Napoli,}
{\em Higher differentiability of minimizers of
 variational integrals with Sobolev coefficients}.
Adv. Calc. Var. 7 (2014), no. 1, 59-89.

\bibitem{APdN2} {A.~Passarelli di Napoli,} {\em Higher differentiability of solutions of elliptic systems
with Sobolev coefficientes: the case $p=n=2$}. Potential Anal. 41
(2014), no. 3, 715-735

\bibitem{APdN-Levico}{A.~Passarelli di Napoli}, {\em Regularity results for non-autonomous variational integrals with discontinuous coefficients}. Rend. Lincei Mat. Appl. 26 (2015),  no. 4, 475--496
 
\bibitem{PV}
\newblock A.~Passarelli di Napoli and  A.~Verde, 
\newblock \emph{ A regularity result for asymptotically convex problems with lower order terms.} 
\newblock J. Convex Anal. 15 (2008), 131--148. 

\bibitem{R}
\newblock J.P.~Raymond,
\newblock \emph{Lipschitz regularity for some asymptotically convex problems.}
\newblock  Proc. Roy. Soc. Edinburgh Sect. A 117 (1991), 59-73. 

 

 \bibitem{scheven1}
 \newblock C.~Scheven and T.~Schmidt,
 \newblock  \emph{Asymptotically regular problems. II. Partial Lipschitz continuity and a singular set of positive measure.}
 \newblock Ann. Sc. Norm. Super. Pisa Cl. Sci.   8 (2009), no. 5, 469--507. 

\bibitem{scheven2}
\newblock C.~Scheven and T.~Schmidt,
\newblock \emph{Asymptotically regular problems. I. Higher Integrability.}
\newblock J. Differential Equations. 248 (2010), 745--791.

\end{thebibliography}
\end{document}